\documentclass[11pt,reqno]{amsart}
\usepackage{amsmath,amscd,amssymb,amsfonts,amsthm,courier,relsize,bm}
\usepackage{hyperref,enumitem,mathrsfs,mathtools}
\usepackage{extarrows}
\usepackage{comment}
\usepackage[percent]{overpic} 

\usepackage{xcolor}  
\hypersetup{
    colorlinks,
    linkcolor={red!50!black},
    citecolor={blue!70!black},
    urlcolor={blue!80!black}
}

\newtheorem{theorem}{Theorem}[section]
\newtheorem{corollary}[theorem]{Corollary}
\newtheorem{conjecture}[theorem]{Conjecture}

\newtheorem{lemma}[theorem]{Lemma}
\theoremstyle{definition}    
\newtheorem{definition}[theorem]{Definition}
\theoremstyle{remark}

\newtheorem{remark}[theorem]{Remark}

\newtheorem{example}[theorem]{Example}

\newcommand{\pair}[2]{\langle #1, #2 \rangle}
\newcommand{\ignore}[1]{}

\newcommand{\ol}[1]{\overline{#1}}

\newcommand{\ti}[1]{\widetilde{#1}}

\newcommand{\scr}[1]{\mathscr{#1}}
\newcommand{\mf}[1]{\mathfrak{#1}}

\newcommand{\tn}[1]{\textnormal{#1}}
\renewcommand{\i}{{\mathrm{i}}}

\def\APS{\ensuremath{\textnormal{APS}}}

\def\ind{\ensuremath{\textnormal{index}}}

\def\Ad{\ensuremath{\textnormal{Ad}}}

\def\kg{\ensuremath{\mathfrak{g}}}

\def\A{\ensuremath{\mathcal{A}}}

\def\D{\ensuremath{\mathcal{D}}}

\def\F{\ensuremath{\mathcal{F}}}

\def\L{\ensuremath{\mathcal{L}}}

\def\bC{\ensuremath{\mathbb{C}}}
\def\bR{\ensuremath{\mathbb{R}}}
\def\bZ{\ensuremath{\mathbb{Z}}}

\def\End{\ensuremath{\textnormal{End}}}
\def\Hom{\ensuremath{\textnormal{Hom}}}

\def\ker{\ensuremath{\textnormal{ker}}}

\def\id{\ensuremath{\textnormal{id}}}

\def\even{\ensuremath{\textnormal{even}}}
\def\odd{\ensuremath{\textnormal{odd}}}

\def\dim{\ensuremath{\textnormal{dim}}}

\def\dom{\ensuremath{\textnormal{dom}}}

\def\pt{\ensuremath{\textnormal{pt}}}

\def\Ch{\ensuremath{\textnormal{Ch}}}
\def\Td{\ensuremath{\textnormal{Td}}}

\def\Pin{\ensuremath{\tn{Pin}}}
\def\Spin{\ensuremath{\tn{Spin}}}

\begin{document}

\title{Geometric quantization of \MakeLowercase{$b$}-symplectic manifolds}

\author{Maxim Braverman}
\address{Department of Mathematics,
Northeastern University,
Boston, MA 02115,
USA}

\email{m.braverman@northeastern.edu}

\author{Yiannis Loizides}
\address{Department of Mathematics,
Pennsylvania State University,
University Park, PA 16802,
USA}

\email{yxl649@psu.edu}

\author{Yanli Song}
\address{Department of Mathematics,
Washington University,
St. Louis, MO 63130,
USA}

\email{yanlisong@wustl.edu}


\title{Geometric quantization of $b$-symplectic manifolds}

\sloppy

\begin{abstract}
We introduce a method of geometric quantization for compact $b$-symplectic manifolds in terms of the index of an Atiyah-Patodi-Singer (APS) boundary value problem.  We show further that b-symplectic manifolds have canonical Spin-c structures in the usual sense, and that the APS index above coincides with the index of the Spin-c Dirac operator. We show that if the manifold is endowed with a Hamiltonian action of a compact connected Lie group with non-zero modular weights, then this method satisfies the Guillemin-Sternberg ``quantization commutes with reduction'' property. In particular our quantization coincides with the formal quantization defined by Guillemin, Miranda and Weitsman, providing a positive answer to a question posed in their paper.
\end{abstract}

\maketitle


\setcounter{tocdepth}{1}

\section{Introduction}
A {\em $b$-symplectic} (or \emph{log symplectic}) manifold $(M,\omega)$ is a manifold $M$ equipped with a `symplectic form' $\omega$ allowed to have a pole along a hypersurface $Z\subset M$.  The study of such manifolds was initiated by Nest and Tsygan \cite{NestTsygan96}, who presented a deformation quantization scheme for such manifolds. A comprehensive study of $b$-symplectic manifolds was conducted in \cite{guillemin2014symplectic,guillemin2014convexity, gualtieri2014symplectic,guillemin2014toric}. 

In \cite{guillemin2018geometric} Guillemin, Miranda and Weitsman  defined the {\em formal geometric  quantization} of a $b$-symplectic  manifold endowed with a Hamiltonian action of a torus $T$. Specifically, they associated to such a manifold a virtual representation of $T$, which satisfies the ``quantization commutes with reduction" property, see \cite{Weitsman01-QNC} for the notion of formal quantization. One of the most surprising results of \cite{guillemin2018geometric} is that the formal quantization space is finite dimensional.  The authors of \cite{guillemin2018geometric} posed a problem of finding a $T$-equivariant  Fredholm operator, whose index space is isomorphic to their formal quantization. The purpose of this paper is to introduce approaches to geometric quantization of $b$-symplectic manifolds using Fredholm operators. 

Let $(M,\omega)$ be a $b$-symplectic manifold  (no group action is required) and let $Z\subset M$ be the exceptional hypersurface along which $\omega$ is singular.  Let $L \to M$ be a prequantum line bundle over $M$, cf. Definition~\ref{D:prequantum}. We construct a compact manifold with boundary $M_\epsilon$ by cutting near the hypersurface $Z$, whose boundary is diffeomorphic to two copies of $Z$. Let $D_{\APS}$ denote the Dolbeault-Dirac operator $D$ acting on sections of $L|_{M_\epsilon}$ and endowed with the Atiyah-Patodi-Singer boundary conditions. Then $D_{\APS}$ is Fredholm and the Fredholm index $\ind(D_{\APS})$ is independent of the cutting (as long as the cutting is close enough to $Z$). We define this index to be the {\em} geometric quantization of $(M,Z,\omega)$:
\begin{equation}\label{E:Iquant=ind}
	Q(M,Z,\omega) \ := \ \ind (D_{\APS} )\in \bZ.
\end{equation}

We can remove the singularity by gluing the manifold $M_\epsilon$ into a closed manifold $\widetilde{M}$ (diffeomorphic to $M$) since it has two isomorphic boundaries. Though the almost complex structure cannot be glued (nor the symplectic form), we nevertheless show that there exists a smooth Spin$^c$-structure on the manifold $\widetilde{M}$, which is canonical up to isomorphism. Moreover we prove that $Q(M, Z, \omega)$ equals the Fredholm index of the Spin$^c$-Dirac operator $\widetilde{D}$ on $\widetilde{M}$, that is  
\begin{equation}
\label{eqn:Qind}
Q(M,Z,\omega) \ = \ \ind( \widetilde{D} )\in \bZ.
\end{equation}
To provide a simple analogy with volumes of b-symplectic manifolds: we think of \eqref{E:Iquant=ind} (or rather its limit as $\epsilon \rightarrow 0^+$) as an analogue of the b-symplectic volume defined by a limit:
\begin{equation} 
\label{eqn:lim}
\lim_{\epsilon\rightarrow 0^+} \int_{M_\epsilon} \frac{\omega^n}{n!}.
\end{equation}
On the other hand $[\omega^n/n!]$ is a top-degree class in the b-de Rham cohomology of $M$.  There is a map (of Mazzeo-Melrose) to the ordinary top-degree de Rham cohomology of $M$, and pairing with $[M]$ gives a number that coincides with \eqref{eqn:lim}.  We think of the index of $\ti{D}$ in \eqref{eqn:Qind} as an analogue of this second description of the b-symplectic volume.

Suppose now that a compact Lie group $G$ acts on $(M,Z, \omega)$ in a Hamiltonian fashion. Then the operators $D_\APS$, $\ti{D}$ are $G$-equivariant, the definition \eqref{E:Iquant=ind} has a refined meaning as a virtual representation of $G$, and \eqref{eqn:Qind} still holds, now as an equality in the representation ring $R(G)$. We conjecture that \eqref{E:Iquant=ind} satisfies ``quantization commutes with reduction'' provided that the reduced space is still a $b$-symplectic manifold (orbifold):
\begin{conjecture}
$[Q(M, Z, \omega)]^G = Q(M_{\text{red}}, Z_{\text{red}},  \omega_{\text{red}}) \in \bZ$.
\end{conjecture}

In Section~\ref{SS:qr} we show that the above conjecture is true if the modular weights of the action are non-zero (see Definition~\ref{def:modularweight}).  For the special case $G=T$ a torus, this is the situation considered by Guillemin, Miranda and Weitsman \cite{guillemin2018geometric}, and hence $Q(M,Z,\omega)$ is isomorphic to the formal quantization defined in \emph{loc.cit.}, providing a positive answer to a question posed in their paper. We will discuss more general cases (possibly non-zero modular weights) elsewhere.

Here is a brief outline of the contents of the sections.  In Section \ref{S:bsymplectic} we review basic properties of b-symplectic manifolds, and in particular describe a normal form for the b-symplectic form and complex structure on $^bTM$ near the hypersurface.  In Section \ref{S:Diracop} we discuss Dirac operators, spin-c structures, and boundary conditions on b-symplectic manifolds, establishing basic properties of $D_\APS$, $\ti{D}$ mentioned above.  In particular we show $\ind(D_\APS)=\ind(\ti{D})$ and provide some equivalent Atiyah-Singer formulas.  In Section \ref{S:qr} we discuss the equivariant situation involving a compact Lie group action with non-zero modular weights, and we show that in this case the $[Q,R]=0$ theorem for \eqref{E:Iquant=ind} follows from work of Ma and Zhang \cite{MaZhangTransEll}.  In the Appendix, we give a conceptual construction of the spin-c structure mentioned above, that illuminates some of its properties; for example, we see that the construction extends to the case with normal crossing singularities.\\

\noindent \textbf{Acknowledgements.} We thank Songhao Li, Reyer Sjamaar and Yi Lin for helpful discussions.

\section{$b$-symplectic manifolds}\label{S:bsymplectic}

In this section we recall some basic properties of $b$-symplectic manifolds. See for example \cite{guillemin2014symplectic} for details.

\subsection{$b$-symplectic forms, defining functions.}\label{SS:bsymplectic}
Let $M^n$ be a closed oriented manifold and let $\iota \colon Z\hookrightarrow M$ be a smooth compact hypersurface.  The $b$ tangent bundle $^bTM$ is the vector bundle over $M$ whose smooths sections consists of vector fields on $M$ that are tangent to $Z$.  This description determines a bracket on the sections as well as a smooth \emph{anchor map}
\[ a \colon {}^bTM \rightarrow TM \] 
given by evaluation, making ${}^bTM$ into a Lie algebroid.  The anchor $a$ is a bundle isomorphism over $M \backslash Z$, whereas $a(^bTM|_Z)=TZ$.   

The $b$ cotangent bundle $^bT^\ast M=(^bTM)^\ast$.  Smooth sections of $^bT^\ast M$ can be thought of as 1-forms allowed to have a logarithmic singularity on $Z$: if $f$ is a smooth function vanishing to order $1$ on $Z=f^{-1}(0)$, then $df/f$ may be thought of as a smooth section of $^bT^\ast M$; for example a section $v \in C^\infty(M,{}^bTM)$ is a vector field on $M$ tangent to $Z$, hence $df(v)$ vanishes to at least order $1$ on $Z$, and $f^{-1}df(v)|_{M\backslash Z}$ extends to a smooth function on $M$.  There is a `$b$ de Rham complex' $(\Omega^\bullet_b(M)=C^\infty(M,\wedge^{\bullet} {}^bT^\ast M),d)$ (in fact a special case of the general definition of the de Rham complex of a Lie algebroid).  A result of Mazzeo-Melrose (see \cite{melrose1993atiyah}, or \cite[Theorem 27]{guillemin2014symplectic}) describes the corresponding $b$-de Rham cohomology groups in terms of ordinary de Rham cohomology groups:
\begin{equation} 
\label{eqn:MazzeoMelrose}
H^p_{b,dR}(M)\simeq H^p_{dR}(M)\oplus H^{p-1}_{dR}(Z).
\end{equation}
Many other notions of ordinary differential geometry have `$b$' analogues, cf. \cite{melrose1993atiyah}.  

\begin{definition}\label{D:bsymplecticform}
A $b$ symplectic form is a closed $b$ 2-form $\omega \in \Omega^2_b(M)$ such that the map of vector bundles $^bTM \rightarrow {}^bT^\ast M$ (induced by contraction) is an isomorphism.
\end{definition}
A $b$ symplectic form is equivalently the `inverse' of an ordinary smooth Poisson structure $\omega^{-1} \in \Gamma(\wedge^2 TM)$ whose top exterior power vanishes transversally along $Z$. 

\begin{remark}
\label{rem:compatiblecomplex}
In particular Definition \ref{D:bsymplecticform} says that $(^bTM,\omega)$ is a symplectic vector bundle.  It therefore admits compatible complex structures $J \in \End(^bTM)$ such that $J^2=-1$, $\omega(Jw,Jv)=\omega(w,v)$ and $\omega(w,Jv)=:g(w,v)$ defines a positive definite inner product on $^bTM$.  (And moreover any two such $J$'s are homotopic.)
\end{remark}

The one-dimensional kernel $\ker(a|_Z)\subset {}^bTM|_Z$ is a line subbundle with a canonical trivialization $e \in C^\infty(Z,\ker(a|_Z))$ locally taking the form
\begin{equation}
\label{eqn:canontriv}
e=yY|_Z
\end{equation}
(restriction in the sense of locally defined sections of $^bTM$) in terms of any (locally defined) function $y$ vanishing to order $1$ on $Z$ and (locally defined) vector field $Y$ satisfying $Yy=1$ (cf. \cite[Section 3]{guillemin2014symplectic}).  The cotangent bundle $T^\ast Z$ to $Z$ is canonically identified with the annihilator of $\ker(a|_Z)=\bR\cdot e$ in $^bT^\ast M|_Z$.  If an ordinary smooth differential form is viewed as a section of $\wedge {}^bT^\ast M$, then its restriction to $Z$ (in the sense of sections of vector bundles) coincides with its pullback to $Z$ (in the sense of differential forms); this is just because sections of $^bTM$ are by definition vectors fields on $M$ which are tangent to $Z$. 

\ignore{
  
Note that for any smooth 1-form $\alpha \in \Omega^1(M)\subset \Omega^1_b(M)$, the restriction $\alpha|_Z \in \Gamma({}^bT^\ast M|_Z)$ as a section of ${}^bT^\ast M$ annihilates $e \in \ker(a|_Z)$, since the vector field $a(e)=x\partial_x$ vanishes on $Z$.  Hence $\alpha|_Z \in \Gamma({}^bT^\ast M|_Z)$ takes values in the subbundle $T^\ast Z$, and so coincides with the pullback $\iota^\ast \tau$ (in the sense of smooth differential forms). This is a slightly surprising point so we reiterate: if a smooth differential form is viewed as a section of $\wedge {}^bT^\ast M$, then its restriction to $Z$ (in the sense of sections of vector bundles) coincides with its pullback to $Z$ (in the sense of differential forms). See for example . }

\begin{definition}\label{D:definingfunction}
A {\em global defining function} for $Z$ is a smooth real-valued function $x\colon M\to \bR$ which vanishes to order $1$ on $Z=x^{-1}(0)$.
\end{definition}
A $b$ symplectic manifold admits a global defining function: indeed one may take $x=\omega^{-n}/\tau$ where $\tau$ is any volume form (recall we assume $M$ oriented).  Hence the orientation on $M$ determines an equivalence class of global defining functions, up to multiplication by smooth positive functions.  This also implies that for an (oriented) $b$ symplectic manifold $(M,Z,\omega)$ the hypersurface $Z$ is necessarily separating, i.e. $M \backslash Z=M_{x>0}\sqcup M_{x<0}$ is disconnected, and $Z$ and its normal bundle are oriented.

\subsection{Normal forms for $\omega$, $J$ near $Z$.}\label{SS:Product}
Let $(M,Z,\omega)$ be a $b$-symplectic manifold. Let
\[ \iota \colon Z \hookrightarrow M \]
denote the inclusion map. The orientation determines an equivalence class of global defining functions for $Z$.  Choose a global defining function $x$ in this class, which we may further normalize such that $|x|=1$ outside a small collar neighborhood
\[ C\simeq Z \times (-1,1), \]
where $x|_C$ is the projection $Z\times (-1,1)\rightarrow (-1,1)$.  Hence $dx|_{M\backslash C}=0$.  Let $\pi \colon C \rightarrow Z$ be the projection.  The function $x$ determines a closed $b$ 1-form $dx/x \in \Omega^1_b(M)$ satisfying $\pair{dx/x}{e}=1$ inducing splittings:
\begin{equation} 
\label{eqn:directsum}
^bT^\ast M|_Z=\bR \frac{dx}{x}\oplus T^\ast Z, \qquad {}^bTM|_Z=\bR e\oplus TZ.
\end{equation}   

Let $\mu$ be any $b$-differential $p$-form.  Then $\mu$ can be written as
\[ \mu =\frac{dx}{x}\wedge \alpha +\beta,\]
for some smooth forms $\alpha \in \Omega^{p-1}(C)$, $\beta \in \Omega^p(M)$ (we allow that $\alpha$ might only be defined on $C$ since $dx=0$ outside of $C$).

The restriction of $\mu$ to $Z$ (as a section of $\wedge {}^bT^\ast M$) is given by (see Section \ref{SS:bsymplectic}):
\begin{equation} 
\label{eqn:decomposemu}
\mu|_Z=\frac{dx}{x}\wedge \iota^\ast \alpha+\iota^\ast \beta.
\end{equation}
It is useful to note in passing that the differential form 
\[ \alpha_Z:=\iota^\ast \alpha \in \Omega^{p-1}(Z) \] 
does not depend on any choices, since it may be expressed as the contraction
\begin{equation} 
\label{eqn:alphacontraction}
\alpha_Z=\iota(e)(\mu|_Z) 
\end{equation}
and the section $e$ of \eqref{eqn:canontriv} does not depend on any choices.  On the other hand $\iota^\ast \beta=\mu|_Z-(dx/x)\wedge \alpha_Z$ depends on the choice of defining function $x$.

\begin{lemma}[cf. Proofs of \cite{guillemin2014symplectic}, Proposition 10, Theorem 27]
\label{lem:decomp}
In the decomposition \eqref{eqn:decomposemu}, one can choose $\alpha=\pi^\ast \alpha_Z$, where $\alpha_Z$ is the intrinsically defined $(p-1)$-form $\alpha_Z$ in \eqref{eqn:alphacontraction}.  Hence there is a decomposition
\[ \mu=\frac{dx}{x}\wedge \pi^\ast \alpha_Z+\beta, \qquad \alpha_Z \in \Omega^{p-1}(Z), \beta \in \Omega^p(M).\]
If $\mu$ is closed then $\alpha_Z$, $\beta$ are closed, and
\[ [\mu] \in H^p_{b,dR}(M)\mapsto ([\alpha_Z],[\beta])\in H^{p-1}_{dR}(Z)\oplus H^p_{dR}(M) \]
realizes the isomorphism \eqref{eqn:MazzeoMelrose}.
\end{lemma}

For the $b$-symplectic form
\begin{equation} 
\label{eqn:firstnormalform}
\omega=\frac{dx}{x}\wedge \pi^\ast \alpha_Z+\beta
\end{equation}
we can say slightly more:
\begin{lemma}
\label{lem:prodformomega}
One can choose $C$, $x$, $\pi$ such that, in the expression \eqref{eqn:firstnormalform} for the $b$-symplectic form, the closed $2$-form $\beta$ has the following property: on a (possibly smaller) collar neighborhood $C'$ of $Z$, $\beta|_{C'}=(\pi^\ast \iota^\ast \beta)|_{C'}$.
\end{lemma}
\begin{proof}
Suppose $\omega$ is decomposed as in \eqref{eqn:firstnormalform}, and let $\beta_Z=\iota^\ast \beta$.  Consider the following $b$ $2$-form, defined on $C$:
\[ \omega_0=\frac{dx}{x}\wedge \pi^\ast \alpha_Z+\pi^\ast \beta_Z. \]
It is closed, as $\alpha_Z$, $\beta_Z$ are both closed.  Using \eqref{eqn:decomposemu}, one has $\omega_0|_Z=\omega|_Z$ as sections of $\wedge^2({}^bT^\ast M|_Z)$, hence $\omega_0$ is nondegenerate along $Z$, and hence on an open collar neighborhood of $Z$.  Replacing $C$ by this smaller collar neighborhood, we may as well assume $\omega_0$ is $b$-symplectic on $C$.

Since $\omega_0|_Z=\omega|_Z \in \Gamma(\wedge^2(^bT^\ast M|_Z))$, we may apply the b-Moser theorem \cite[Theorem 34]{guillemin2018geometric} of Guillemin-Miranda-Pires on $(C,Z)$ and to the $b$-symplectic forms $\omega|_{C},\omega_0$.  This provides tubular neighborhoods $C_0,C_1$ of $Z$ in $C$ and a diffeomorphism $\varphi \colon C_1\rightarrow C_0$ such that $\omega|_{C_1}=(\varphi^\ast \omega_0)|_{C_1}$ and $\varphi|_Z=\id_Z$.  In fact $\varphi$ is the result of integrating a time-dependent vector field $v_t$ that vanishes on $Z$.  If we multiply $v_t$ by a suitable bump function supported in a small neighborhood of $Z$, we obtain a diffeomorphism which still satisfies $\omega|_{C_1'}=(\varphi^\ast \omega_0)|_{C_1'}$ for possibly smaller tubular neighborhoods $C_0',C_1'$ of $Z$, and is the identity outside a small neighborhood of $Z$, so admits an extension by the identity to $M$, which we also denote by $\varphi$. So $\varphi$ now denotes a global diffeomorphism $(M,Z)\rightarrow (M,Z)$ having the property that 
\[ \omega|_{C_1'}=(\varphi^\ast \omega_0)|_{C_1'}.\]

Let $\ti{x}=\varphi^\ast x$, $\ti{\pi}=\pi\circ \varphi|_{\ti{C}}$ where $\ti{C}=\varphi^{-1}(C)$.  Then $\ti{x}$ is a new defining function for $Z$ and $\ti{\pi}\colon \ti{C}\rightarrow Z$ is a new collar neighborhood.  Moreover, on $\ti{C}$
\[ \varphi^\ast \pi^\ast \alpha_Z=\ti{\pi}^\ast \alpha_Z\]
and likewise for $\pi^\ast \beta_Z$.  Hence on $C_1'\subset \ti{C}$,
\[ \omega|_{C_1'}=(\varphi^\ast \omega_0)|_{C_1'}=\frac{d\ti{x}}{\ti{x}}\wedge \ti{\pi}^\ast\alpha_Z+\ti{\pi}^\ast\beta_Z.\]
\end{proof}

In the rest of this subsection, we will assume $C$, $\pi$, $x$ are chosen as in Lemma \ref{lem:prodformomega}, thus
\begin{equation} 
\label{eqn:decomposeomega}
\omega=\frac{dx}{x}\wedge \pi^\ast \alpha_Z+\beta 
\end{equation}
and $\beta$ equals $\pi^\ast \iota^\ast \beta$ on a small collar neighborhood $C'$ of $Z$.

With a decomposition as in \eqref{eqn:decomposeomega} fixed, let $e'$ be the vector field on $Z$ determined by the equations
\[ \iota(e')\alpha_Z=1, \qquad \iota(e')\iota^\ast \beta=0.\]
Via the splitting \eqref{eqn:directsum}, $e'$ is identified with a section of $^bTM|_Z$, $\{e,e'\}$ is linearly independent, hence
\[ E=\tn{span}\{e,e'\} \subset {}^bTM|_Z \]
is a trivial real rank $2$ vector subbundle over $Z$.  Let 
\[ \omega_E=\omega|_E=\frac{dx}{x}\wedge \alpha_Z. \]  
Then $(E,\omega_E)$ is a symplectic vector bundle.  Define a complex structure $J_E$ on $E$ by
\begin{equation}
\label{E:almostcomplex}
J_E e=e', \quad  J_Ee'=-e.
\end{equation}
We obtain a compatible metric $g_E(\cdot,\cdot)=\omega_E(\cdot,J_E\cdot)$ on $E$, such that $\{e,e'\}$ is a global orthonormal frame, hence in terms of the dual frame $dx/x$, $\alpha_Z|_E$,
\[ g_E=\frac{dx^2}{x^2}+\alpha_Z^2 \in \tn{Sym}^2(E^\ast). \]

There is a direct sum decomposition
\begin{equation} 
\label{eqn:directsum2}
^bTM|_Z=E \oplus F, \qquad F=\ker(\alpha_Z). 
\end{equation}
In fact $F=T\F$ is the tangent distribution for the foliation $\F$ of $Z$ into symplectic leaves for the Poisson structure $\omega^{-1}$, and 
\[ \omega_F:=(\iota^\ast\beta)|_F \]
is the family of symplectic forms on the leaves \cite[Proposition 10]{guillemin2018geometric}.  In particular, $(F,\omega_F)$ is a symplectic vector bundle.  Let $J_F$ be a compatible complex structure on $F$, that is $\omega_F(J_Fw,J_Fv)=\omega_F(w,v)$ and
\[ g_F(w,v):=\omega_F(w,J_Fv)\] 
is a positive definite metric on $F$.

By construction, the direct sum decomposition \eqref{eqn:directsum2} identifies $\omega|_Z$ with $\omega_E\oplus \omega_F$.  Thus 
\[ J=J_E\oplus J_F, \qquad g=g_E\oplus g_F \]
defines a compatible complex structure and induced metric on the symplectic vector bundle $(E\oplus F={}^bTM|_Z,\omega|_Z)$, and moreover
\begin{equation} 
\label{eqn:gprodform}
g=\frac{dx^2}{x^2}+\alpha_Z^2+g_F.
\end{equation}
Using the product structure on $C'\subset C\simeq Z \times (-1,1)$, we identify the pull backs $\pi^\ast E$, $\pi^\ast F$ with complementary subbundles of $^bTC'$. Using the fact that $\beta|_{C'}=(\pi^\ast \iota^\ast \beta)|_{C'}$ (Lemma \ref{lem:prodformomega}), $\omega|_{C'}$ is identified with $\pi^\ast \omega_E\oplus \pi^\ast \omega_F$.  Hence $\pi^\ast J$ gives a compatible complex structure for the symplectic vector bundle $(^bTC',\omega|_{C'})$, with induced metric $\pi^\ast g$ taking the form \eqref{eqn:gprodform} on $C'$. 

One can obtain globally defined compatible $g$, $J$ which take the product form described above near $Z$: let $g_C$, $J_C$ now denote the metric and complex structure on $^bTC$ constructed above, let $\rho$ be a bump function equal to $1$ near $Z$ with support in $C$, and let $g_{M\backslash Z}$ be an arbitrary metric on $M\backslash Z$.  If $g_1:=\rho g_C+(1-\rho) g_{M\backslash Z}$, one obtains an invertible anti-symmetric $A \in \End(^bTM)$ such that $g(\cdot,\cdot)=\omega(\cdot,A\cdot)$.  Then $J:=A/|A|$ equals $J_C$ near $Z$ and $g(\cdot,\cdot):=\omega(\cdot,J\cdot)$ has the desired properties. We summarize the results of this section with a lemma:

\begin{lemma}[Normal forms for $\omega$, $J$, $g$ near $Z$]
Let $(M,Z,\omega)$ be a compact oriented $b$-symplectic manifold.  Then there is a global defining function $x$ and collar neighborhood $\pi \colon C\simeq Z \times (-1,1) \rightarrow Z$ with $|x|=1$ outside $C$, such that
\[ \omega=\frac{dx}{x}\wedge \pi^\ast \alpha_Z + \beta,\]
where $\alpha_Z \in \Omega^1(Z)$, $\beta \in \Omega^2(M)$ are closed and $\beta|_{C'}=(\pi^\ast \iota^\ast \beta)|_{C'}$ on a (possibly smaller) collar neighborhood $C'$ of $Z$.  There exists a direct sum decomposition
\[ ^bTM|_{C'}={}^bTC' \simeq \pi^\ast E \oplus \pi^\ast F \]
where $F=\ker(\alpha_Z)$, $E=\tn{span}\{e,e'\}$ are subbundles of $^bTM|_Z$, $e'$ being uniquely determined by $\alpha_Z(e')=1$, $\iota^\ast\beta(e')=0$ and $e$ the canonical section in \eqref{eqn:canontriv}.  There exists a compatible complex structure $J$ on the symplectic vector bundle $(^bTM,\omega)$ such that the metric $g(\cdot,\cdot)=\omega(\cdot,J\cdot)$ takes a product form near $Z$:
\[ g|_{C'}=\frac{dx^2}{x^2}+\pi^\ast \alpha_Z^2+\pi^\ast g_F \]
where $g_F$ is a compatible metric on the symplectic vector subbundle $(F,\iota^\ast \beta|_F)\subset (^bTM|_Z,\omega|_Z)$.
\end{lemma}

\subsection{Prequantization}\label{SS:integrability}
A symplectic manifold is prequantizable if the de Rham cohomology class of its symplectic form lies in the image of the map $H^2(M,\bZ)\rightarrow H^2_{dR}(M)$.  A prequantization is then a choice of integral lift. For a $b$-symplectic manifold $(M,Z,\omega)$ a slightly more sophisticated definition of prequantization (or integrality) was suggested in \cite[Section 2 and Section 5]{guillemin2018geometric}.  Recall from Lemma \ref{lem:decomp} that there are canonical de Rham cohomology classes $([\alpha_Z],[\beta])\in H^1_{dR}(Z)\oplus H^2_{dR}(M)$ associated to $(M,Z,\omega)$.

\begin{definition}[\cite{guillemin2018geometric}]
\label{D:integrability}
$(M,Z,\omega)$ is \emph{prequantizable} if $([\alpha_Z], [\beta]) \in H^1_{dR}(Z)\oplus H^2_{dR}(M)$ lies in the image of the map from $H^1(Z,\bZ)\oplus H^2(M,\bZ)$. A \emph{prequantization} is a choice of integral lift.
\end{definition}
\begin{remark}
The integrality of $[\alpha_Z]$ will not play any role for us below.
\end{remark}

\begin{definition}\label{D:prequantum}
\label{D:logconnection}
Given an integral lift $[\hat{\beta}]$ of $[\beta]$, one can find a complex line bundle $L \to M$ (unique up to isomorphism) whose first Chern class is $[\hat{\beta}]$.  We call $L$ a \emph{prequantum line bundle} for $(M,Z,\omega)$.  Choosing a decomposition $\omega=(dx/x)\wedge \pi^\ast \alpha_Z+\beta$ as in Lemma \ref{lem:decomp}, one can find a Hermitian metric and Hermitian connection $\nabla^{L,0}$ such that $\beta=\tfrac{\i}{2\pi}(\nabla^{L,0})^2$.  The corresponding \emph{log prequantum connection} $\nabla^L$ is the connection on $L|_{M\backslash Z}$ given by
\begin{equation}\label{E:logconnection}
\nabla^L\ = \ \nabla^{L,0} - 2\pi \i \log(|x|) \pi^\ast \alpha_Z. 
\end{equation}
In this case, we have that 
\begin{equation}\label{E:logcurvature}
	\omega|_{M\backslash Z} \ = \ \frac{\i}{2\pi}(\nabla^L)^2.
\end{equation}
\end{definition} 
\begin{remark}
\label{rem:notabconn}
As a small caution, the 1-form $\log(|x|)\pi^\ast \alpha_Z$ on $M \backslash Z$ is not a $b$ 1-form, and \eqref{E:logconnection} is not a $b$ connection in the sense of \cite{melrose1993atiyah}.
\end{remark}

\section{Dirac operators on $b$-symplectic manifolds}\label{S:Diracop}
\subsection{The Dolbeault-Dirac operator}\label{SS:Dirac}
Let $L \to M$ be a prequantum line bundle for $(M,Z,\omega)$ as in Definition~ \ref{D:prequantum}. Fix a compatible complex structure $J$ on the symplectic vector bundle $(^bTM,\omega)$, and let $g$ be the metric on $^bTM$ induced by $(\omega,J)$.  We use $g$ to identify $^bTM\simeq {}^bT^\ast M$. As usual let $^bT^{(1,0)}M$, $^bT^{(0,1)}M$ be the $+\i$, $-\i$ eigenbundles of $J$ in $^bTM\otimes \bC$, and similarly for $^bT^{\ast,(1,0)}M$, $^bT^{\ast,(0,1)}M$.  Define
\begin{equation}\label{E:Epm}
S^{+,b} :=  \bigoplus_{k\, \even} 
	\wedge^k\big({}^bT^{\ast,(0,1)}M\big)\otimes L,
	\qquad
	S^{-,b} :=  \bigoplus_{k\, \odd} 
	\wedge^k\big({}^bT^{\ast,(0,1)}M\big)\otimes L,
\end{equation}
then $S^b:=S^{+,b}\oplus S^{-,b}$ is a $\bZ_2$-graded complex Hermitian vector bundle over $M$ (the Hermitian structure is induced from $\omega$, $g$).  There is a standard Clifford action (cf. \cite{LawsonMichelsohn} or the Appendix for background) of $\bC l({}^bTM)=\bC l({}^bTM,g)$ on $S$ given by 
\[
	c(v) \,s := \ 
	\sqrt{2}\,\big(\, v^{0,1}\wedge s - \iota(v^{1,0})s\,\big),
	\qquad v\in {}^bTM, \ s \in S.
\]
Using the symplectic orientation (equivalently, the orientation compatible with the complex structure) on ${}^bTM$, we define the \emph{chirality element} at a point $m \in M$ in terms of an oriented orthonormal frame $v_1,...,v_n$ for ${}^bT_mM$ by
\begin{equation} 
\label{eqn:chirality1}
\Gamma_m=\i^{n/2} v_1\cdots v_n \in \bC l({}^bT_mM).
\end{equation}
It is independent of the choice of oriented frame, hence determines a globally defined section $\Gamma$ of $\bC l({}^bTM)$ satisfying $\Gamma^2=1$.  The $+1$ (resp. $-1$) eigenbundle of $\Gamma$ is $S^{b,+}$ (resp. $S^{b,-}$).  In this sense one says that the $\bZ_2$-grading on $S^b$ is \emph{compatible} with the symplectic orientation on $(^bTM,\omega)$.

Via the anchor, the restriction of $g$ to $M \backslash Z$ becomes a complete Riemannian metric in the usual sense.  Let
\[ L^2(M\backslash Z,S^b)= L^2(M\backslash Z,S^{+,b})\oplus L^2(M\backslash Z,S^{-,b}) \]  
be the $\bZ_2$-graded Hilbert space of square-integrable sections of $S^b|_{M\backslash Z}$ (with respect to the Riemannian volume determined by $g$ and using the Hermitian metric on $S^b$). 

Let $\nabla^{J}$ be the connection on $\wedge^\bullet\big({}^bT^{\ast,(0,1)}(M\backslash Z)\big)$ obtained by projecting the Levi-Civita connection. Together with the prequantum connection $\nabla^L$, we obtain a connection on $S|_{M\backslash Z}$ given by 
\[
\nabla = \nabla^J \otimes 1 + 1 \otimes \nabla^L, 
\]
satisfying
\[
[\nabla_v, c(w)] = c(\nabla^{\text{LC}}_v w), \hspace{5mm} v, w \in T(M\backslash Z),
\]
where $\nabla^{\text{LC}}$ is the Levi-Civita connection on $T(M\backslash Z)$. The associated Dirac operator is given by the composition
\begin{equation}
\label{E:Diracoperator}	
D^b\colon C^\infty_c(M\backslash Z, S^b)  \xlongrightarrow{\nabla}  C^\infty_c(M\backslash Z, T^*(M\backslash Z) \otimes S^b) \xlongrightarrow{c \circ g^\sharp}  C^\infty_c(M\backslash Z, S^b).
\end{equation}
Then $D^b$ is an odd essentially self-adjoint first-order elliptic differential operator on the complete manifold $M\backslash Z$, cf. \cite[Th.~1.17]{GromovLawson}. We denote by $D^{b,\pm}$ the components of $D^b$ mapping $C^\infty_c(M\backslash Z,S^{b,\pm})\rightarrow C^\infty_c(M\backslash Z,S^{b,\mp})$. 

\begin{definition}\label{D:quantizationregular}
If $Z = \emptyset$, then $M$ is compact and $D=D^b$ is \emph{Fredholm} in the sense that $\ker(D^\pm)$ are finite dimensional. We define the \emph{geometric quantization} of $(M, \omega)$ to be 
\[
Q(M, \omega) := 
\ind(D) =\dim\, \ker(D^+) -\dim\, \ker(D^-) \in \bZ. 
\]
Compare e.g. \cite{MeinrenkenSymplecticSurgery}. By the Atiyah-Singer index theorem, we have the formula:
\[
Q(M, \omega)  = \int_M \Td(M) \Ch(L)=\int_M \Td(M)e^\omega. 
\]
\end{definition}
When $Z \neq \emptyset$, $D^b$ is not Fredholm in general. In order to generalize Definition~\ref{D:quantizationregular} to $b$-symplectic manifolds, we will introduce boundary conditions in the next subsection.

\subsection{APS boundary conditions and the definition of $Q(M,Z,\omega)$}\label{SS:bc}
Let $C \simeq Z \times (-1,1)$ be a collar neighborhood of $Z$ as in Section~\ref{SS:Product}.  Recall from Lemma \ref{lem:prodformomega} that we may assume the metric takes a product form on a possibly smaller collar $C' \subset C$, and without loss of generality we assume $C'\simeq Z \times (-\tfrac{1}{2},\tfrac{1}{2})$. Let $0< \epsilon <\tfrac{1}{2}$. Cutting $M$ along $Z_{\pm \epsilon}:=Z\times \{\pm\epsilon\}$, we obtain a compact manifold with boundary:
\[
M_\epsilon:= M \backslash \big(Z\times (-\epsilon,\epsilon)\big)=M_{\ge \epsilon} \sqcup M_{\le -\epsilon}, \qquad \partial M_{\ge \epsilon}=Z_\epsilon, \quad M_{\le -\epsilon}= Z_{-\epsilon}, 
\]
where the notation is self-explanatory.
\begin{center}
\begin{overpic}[height=5cm]{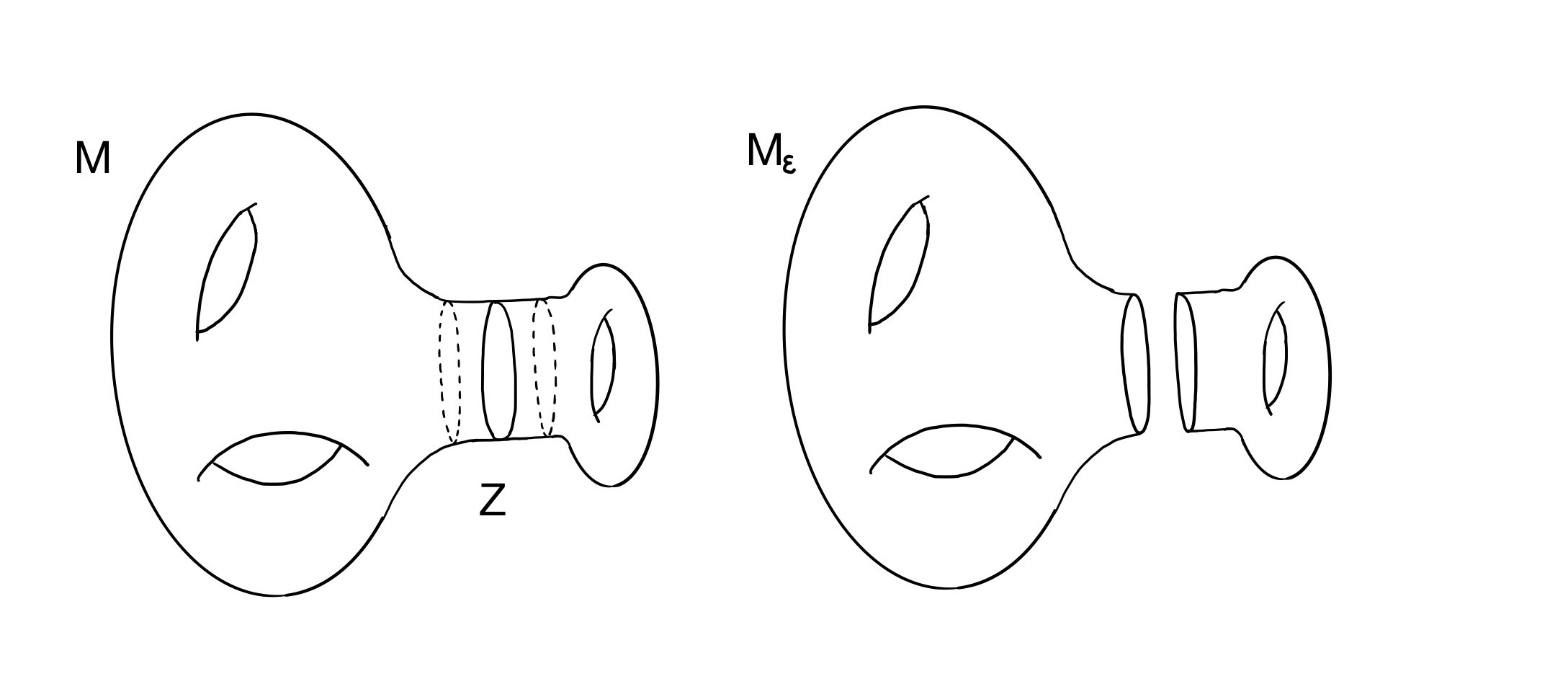}
\end{overpic}\\
Figure 1. The cutting surgery
\end{center}

We will describe a Fredholm boundary value problem on $M_\epsilon$, whose index will generalize Definition \ref{D:quantizationregular}. First it is necessary to modify the grading on $S^b|_{M_\epsilon}$, so that it becomes compatible with the global orientation on $M$. Indeed recall that the grading on $S^b|_{M\backslash Z}$ is compatible with the orientation on $M\backslash Z$ induced by the symplectic form $\omega$.  This differs from the global orientation on $M$ by the locally constant function $\tn{sign}(x)$ on $M\backslash Z$.
\begin{definition}
Let $S=S^b|_{M_\epsilon}$ but introduce a new $\bZ_2$-grading, by reversing the grading on $M_{\le -\epsilon}$:
\[ S^{\pm}\upharpoonright M_{\ge \epsilon}:=S^{b,\pm}, \qquad S^{\pm} \upharpoonright M_{\le -\epsilon}:=S^{b,\mp}.\]
Let $D=D^b$ viewed as an odd operator for the new grading, so that 
\[ D^\pm=(D^{b,\pm}\upharpoonright M_{\ge \epsilon})\sqcup (D^{b,\mp}\upharpoonright M_{\le -\epsilon}).\] 
\end{definition}

We will write $e$ for the section of $^bTM|_C$ such that $a(e)=x\partial_x$; this is consistent with earlier notation because $e$ restricts to the canonical section \eqref{eqn:canontriv} along $Z$. On $C' \cap M_\epsilon$ the operator $D$ takes the form
\begin{equation} 
\label{eqn:DonC}
D^+|_{C'}=c(e)(x\partial_x+A_x).
\end{equation}
Here $A_x$ is the smooth family of essentially self-adjoint first-order elliptic operators on the compact manifold $Z$ given by
\begin{equation} 
\label{eqn:Ax}
A_x=c(e)^{-1}\big(D_Z-2\pi \i \log(|x|)c(e')\big) 
\end{equation}
where $D_Z$ is a Dirac operator on $Z$ not depending on $x$ (defined using the connection $\nabla^{L,0}$ on $L|_Z$, and having symbol $\sigma(D_Z)(\xi)=c(g^\sharp(\xi))$, $\xi \in T^\ast Z$) and recall $e'$ is the section metrically dual to $\alpha_Z$ (see Lemma \ref{lem:prodformomega}).  (The expression \eqref{eqn:DonC} may look more familiar in terms of the coordinate $t=\log(|x|)$.)

We next explain the type of Atiyah-Patodi-Singer \cite{AtiyahPatodiSingerI} (APS) boundary condition that we will use. Aside from the original paper \cite{AtiyahPatodiSingerI}, a concise guide to elliptic boundary value problems requiring relatively modest background knowledge is \cite{BarBallmannGuide} (the monograph \cite{BarBallmann} containing the corresponding technical details). 

The operators $A_{\pm \epsilon}$ are essentially self-adjoint first-order elliptic operators on $Z_{\pm \epsilon}$.  Since the latter are compact manifolds, $A_{\pm \epsilon}$ have discrete spectrum.  Let
\[ C^\infty(M_{\ge \epsilon},S^+)_{A_\epsilon<0} \subset C^\infty(M_{\ge \epsilon},S^+) \]
be the subspace of sections whose restriction to the boundary lies in the closure of the direct sum of the eigenspaces of $A_\epsilon$ with eigenvalue $\lambda < 0$.  One similarly defines $C^\infty(M_{\ge \epsilon},S^+)_{A_\epsilon\le 0}$, as well as similar subspaces for $M_{\le -\epsilon}$ using the operator $A_{-\epsilon}$ in place of $A_\epsilon$. We also set
\[ C^\infty(M_\epsilon,S^+)_{<0}=C^\infty(M_{\ge \epsilon},S^+)_{A_\epsilon<0}\oplus C^\infty(M_{\le -\epsilon},S^+)_{A_{-\epsilon}<0} \subset C^\infty(M_\epsilon,S^+).\]

Suppose first for simplicity that $A_{\epsilon}$ does not have $0$ in its spectrum.  In fact by \eqref{eqn:Ax} this happens if and only if $2\pi \i \log(|\epsilon|)$ is not in the spectrum of the fixed ($\epsilon$-independent) operator $c(e')^{-1}D_Z$, and since the latter has discrete spectrum, this is guaranteed to occur for generic $\epsilon$.  In this case we make the following definition.  The Hilbert space $L^2(M_\epsilon,S^\pm)$ is defined just like $L^2(M\backslash Z,S^\pm)$, using the restrictions of the metrics to $M_\epsilon \subset M\backslash Z$.

\begin{definition}
Suppose $0 \notin \tn{spec}(A_\epsilon)$.  We define the unbounded Hilbert space operator $D_{APS,\epsilon}^+$ to be the closure of the operator with domain $C^\infty(M_\epsilon,S^+)_{<0}$ given by $D_{APS,\epsilon}^+s=D^+s$ for all $s \in C^\infty(M_\epsilon,S^+)_{<0}$.  This is an example of \emph{Atiyah-Patodi-Singer boundary conditions} \cite{AtiyahPatodiSingerI}. 
\end{definition}

More generally if $A_\epsilon$ is not necessarily invertible, then we prefer a slightly more complicated choice (having better continuity properties):
\begin{definition}
\label{def:DAPS}
For any $\epsilon$, we define the unbounded Hilbert space operator $D_{APS,\epsilon}^+$ to be the closure of the operator with initial domain 
\[ \D^\infty_{APS}=C^\infty(M_{\ge \epsilon},S^+)_{A_\epsilon\le 0}\oplus C^\infty(M_{\le -\epsilon},S^+)_{A_{-\epsilon}<0} \] 
given by $D_{APS,\epsilon}^+s=D^+s$ for all $s \in \D^\infty_{APS}$.  In other words, we use the APS boundary condition on $M_{\le -\epsilon}$ and the `dual' APS boundary condition on $M_{\ge \epsilon}$. (`Dual' refers to the fact that the closure of the operator with domain $C^\infty(M_{\ge \epsilon},S^+)_{A_\epsilon\le 0}$ ends up being the adjoint of the operator with domain $C^\infty(M_{\ge \epsilon},S^-)_{A_\epsilon<0}$.)  We refer to these as `APS-like' boundary conditions, for short; they coincide with APS boundary conditions for generic $\epsilon$.
\end{definition}

It is known that the domain $\dom(D_{APS,\epsilon}^+)$ of the closure is the subspace of the Sobolev space $H^1(M_\epsilon,S^+)$ satisfying the boundary conditions (restriction to the boundary being replaced by the trace map $H^1(M_\epsilon,S^+)\rightarrow H^{1/2}(\partial M_\epsilon,S^+)$ in Definition \ref{def:DAPS}), and moreover that
\[ D_{APS,\epsilon}^+\colon \dom(D_{APS,\epsilon}^+)\rightarrow L^2(M_\epsilon,S^-) \]
is a Fredholm operator, cf. \cite{BarBallmannGuide}.

Our generalization of Definition \ref{D:quantizationregular} is as follows.
\begin{definition}
\label{def:geoquant}
We define the \emph{geometric quantization} of the compact $b$-symplectic manifold $(M,Z,\omega)$ to be the index
\[ Q(M,Z,\omega):=\ind(D_{APS,\epsilon}^+) \in \bZ,\]
where $D_{APS,\epsilon}^+$ is the operator on $M_\epsilon$ with APS-like boundary conditions in Definition \ref{def:DAPS}.  We will justify that $Q(M,Z,\omega)$ is independent of $\epsilon$ in Theorem \ref{thm:indexequal} below.  In particular one could equivalently define $Q(M,Z,\omega)$ in terms of a limit of APS-like indices, as $\epsilon \rightarrow 0^+$.
\end{definition}

\subsection{The spin-c structure on a $b$-symplectic manifold}\label{SS:spinor}
The anchor map $a \colon {}^bTM \rightarrow TM$ is a bundle isomorphism over $M \backslash Z$.  In particular the manifold with boundary $M_\epsilon$ has a metric and compatible almost complex structure induced from those on the symplectic vector bundle $({}^bTM,\omega)$.  The boundary $\partial M_\epsilon=Z_\epsilon \sqcup Z_{-\epsilon}$, and one might ask if it is possible to glue the two boundary components $Z_{\pm \epsilon}$ together and obtain a metric and almost complex structure on the manifold (diffeomorphic to $M$)
\[
\widetilde{M} \colon=  M_\epsilon \Big/ (m, \epsilon) \sim (m, -\epsilon), \quad m \in Z.
\]
\begin{center}
\begin{overpic}[height=5cm]{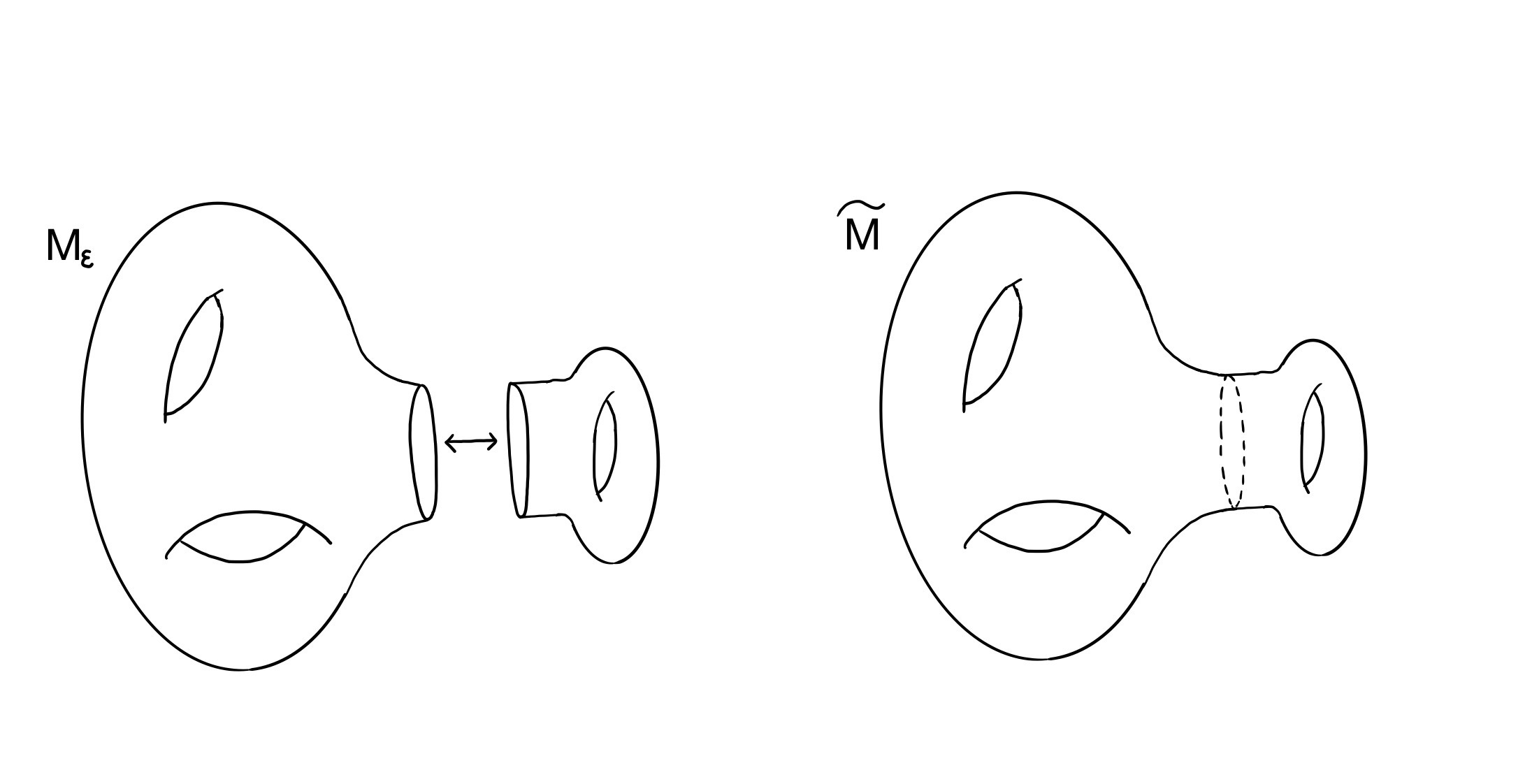}
\end{overpic}\\
Figure 2. The gluing surgery  
\end{center}
By Lemma \ref{lem:prodformomega}, the metric $g$ takes a product form near $Z$, hence can be glued into a metric on $\widetilde{M}$.  However for the almost complex structure (equivalently for the symplectic form), this naive idea does not seem to work, as one can convince oneself by considering the local normal form in Section \ref{SS:Product}. We will show that nevertheless it \emph{is} possible to glue together the corresponding \emph{spin-c structures} to obtain a spin-c structure on the manifold $\ti{M}$ in the ordinary sense.  A more systematic construction illuminating some of the properties of the spin-c structure and avoiding use of local normal forms is given in the appendix.

Denote by $S_{\pm\epsilon}$ the restriction of the spinor bundle $S$ to the boundary component $Z_{\pm \epsilon}$:
\[ S_{-\epsilon}=\wedge {}^bT^{\ast,(0,1)}M \otimes L\big|_{Z_{-\epsilon}}, \qquad S_\epsilon=\wedge {}^bT^{\ast,(0,1)}M\otimes L\big|_{Z_{\epsilon}}.\]
As we have arranged that the metric and complex structure on the vector bundle $^bTM$ take a product form on $C'\simeq Z \times (-\tfrac{1}{2},\tfrac{1}{2})$, we may identify $S_\epsilon$ with $S_{-\epsilon}$ as ungraded $\bC l(^bTM|_Z)$-modules; we will write 
\[ \rho \colon S_\epsilon \rightarrow S_{-\epsilon}\] 
for the corresponding morphism of Clifford modules. Since $S=S^b|_{M\backslash Z}$ but with the $\bZ_2$ grading reversed over $M_{<0}$, the map $\rho$ has \emph{odd} parity. In particular if $z \in \bC l(^bTM|_Z)$ is of odd degree in the Clifford algebra, then $c(z)\circ \rho=\rho \circ c(z)$ maps $S^+_\epsilon$ to $S^+_{-\epsilon}$.

\begin{definition}
We define a smooth grading-preserving bundle map
\[ \tau \colon S_{\epsilon}\rightarrow S_{-\epsilon}, \qquad \tau(s)=c(\Gamma e)\rho(s) \]
where $\Gamma$ is the chirality element \eqref{eqn:chirality1}.  Since $c(\Gamma e)$ commutes with $c(e')$, $D_N$ and anti-commutes with $c(e)^{-1}=-c(e)$, the induced map on smooth sections (also denoted $\tau$) satisfies
\begin{equation} 
\label{eqn:anticommute}
\tau \circ A_\epsilon=-A_{-\epsilon}\circ \tau.
\end{equation}
\end{definition}

\begin{definition}
\label{def:gluetilS}
Define a $\bZ_2$ graded complex vector bundle $\ti{S}$ on $\ti{M}$ by gluing $S_{\epsilon}$ to $S_{-\epsilon}$ using $\tau$:
\[
\widetilde{S} :=  S|_{M_\epsilon} \Big/ (s, \epsilon) \sim (\tau(s), -\epsilon), \qquad s \in S_{\epsilon}.
\]
A continuous section of $\ti{S}$ is by definition a pair $(s_1, s_2)$ of sections over $M_{\ge \epsilon}$, $M_{\le -\epsilon}$ respectively, such that $\tau(s_1|_{Z_\epsilon})=s_2|_{Z_{-\epsilon}}$.  (For a smooth construction, one easily modifies this by gluing on a small collar.)
\end{definition}

Since the anchor map $a \colon {}^bTM \rightarrow TM$ is an isomorphism over $M_\epsilon$, we obtain a $\bC l(TM_\epsilon)$-module structure on $S|_{M_\epsilon}$ given by
\begin{equation} 
\label{eqn:ClmodM}
\ti{c}:=c\circ a^{-1}.
\end{equation}
\begin{lemma}
$\tau$ intertwines the $\bC l(TM_\epsilon)$-module structure \eqref{eqn:ClmodM} along $Z_\epsilon$, $Z_{-\epsilon}$, which therefore descends to a $\bC l(T\ti{M})$-module structure on $\ti{S}$.  Equivalently, $\tau$ defines a trivialization of the line bundle $\Hom_{\bC l(TM|_Z)}(S_\epsilon,S_{-\epsilon})$.
\end{lemma}
\begin{proof}
For $s \in S_\epsilon$, and $v \in TZ$ one has
\[ \ti{c}(v)\tau(s)=c(v)c(\Gamma e)\rho(s)=c(\Gamma e)\rho(c(v)s)=\tau(\ti{c}(v)s) \]
since $v$, $\Gamma e$ commute and $a^{-1}(v)=v$ for $v \in TZ$.  The remaining generator $\partial_x$ requires care, since
\[ \big(a|_{Z_\epsilon})^{-1}(\partial_x)=\tfrac{1}{\epsilon}e, \qquad \big(a|_{Z_{-\epsilon}})^{-1}(\partial_x)=-\tfrac{1}{\epsilon}e \]
differ by a sign.  Thus for $s \in S_\epsilon$
\[ \ti{c}(\partial_x)\tau(s)=-\tfrac{1}{\epsilon}c(e)c(\Gamma e)\rho(s)=c(\Gamma e)\rho(\tfrac{1}{\epsilon}c(e)s)=\tau(\ti{c}(\partial_x)s),\]
where for the second equality we used the fact that $e$, $\Gamma e$ \emph{anti-commute}.
\end{proof}

\begin{definition}
\label{def:associatedspinormodule}
The $\bC l(T\ti{M})$-spinor module $\ti{S}$ will be referred to as the \emph{spin-c structure associated to the b-symplectic manifold} $(M,Z,\omega)$.  Up to homotopy, it does not depend on $\epsilon$.  The \emph{determinant line bundle} (or \emph{anti-canonical line bundle}) of $\ti{S}$ is the complex line bundle
\[ \ti{\L}=\Hom_{\bC l(T\ti{M})}(\ti{S}^\ast,\ti{S}). \]
\end{definition}

\begin{theorem}
\label{thm:indexequal}
$Q(M,Z,\omega)$ equals the index of any spin-c Dirac operator for $\ti{S}$.  In particular $Q(M,Z,\omega)$ is independent of $\epsilon$.
\end{theorem}
\begin{proof}
Recall $Q(M,Z,\omega)=\ind(D_{APS,\epsilon}^+)$ was defined as the index of (the closure of) the operator with domain
\[ \D^\infty_{APS}=C^\infty(M_{\ge \epsilon},S^+)_{A_\epsilon\le 0}\oplus C^\infty(M_{\le -\epsilon},S^+)_{A_{-\epsilon}<0}, \]
given by $D_{APS,\epsilon}^+s=D^+s$, and that near $x=\pm \epsilon$, 
\begin{equation} 
\label{eqn:DonC2}
D|_{C'}=c(e)(x\partial_x+A_x)
\end{equation}
where $A_x$ is the smooth family of first-order elliptic operators on the compact manifold $Z$ given by
\begin{equation} 
\label{eqn:Ax2}
A_x=c(e)^{-1}\big(D_Z-2\pi \i \log(|x|)c(e')\big). 
\end{equation}
We may perform a homotopy of the operator $D$ to replace the family $\{A_x|0\ne x \in (-\tfrac{1}{2},\tfrac{1}{2})\}$ of operators on $Z$ with the constant family of operators $\ti{A}_x=\ti{A}$, $0\ne x \in (-\tfrac{1}{2},\tfrac{1}{2})$ where
\[ \ti{A}=c(e)^{-1}D_Z.\]
This results in a new Dirac operator $\ti{D}$ on $M \backslash Z$ such that
\begin{equation} 
\label{eqn:tilDonC}
\ti{D}|_{C'}=c(e)(x\partial_x+\ti{A})=c(e)x\partial_x+D_Z.
\end{equation}
In doing this we keep the boundary condition fixed (i.e. it is still specified in terms of spectral subspaces for the operators $A_{\pm\epsilon}$); in other words, the domains of the operators in the homotopy do not vary.  The resulting operator $\ti{D}_{APS,\epsilon}^+$ has the same index. 

We claim that the new Dirac operator $\ti{D}$ can be glued smoothly to a spin-c Dirac operator on $\ti{M}$ acting on $\ti{S}$.  To see this we describe a smooth version of the gluing construction in Definition \ref{def:gluetilS}.  In terms of the logarithmic coordinate $x$ that we are using, a suitable gluing diffeomorphism on the base may be given by the identity map on $Z$ and the diffeomorphism (note $2\epsilon^2<\epsilon<\tfrac{1}{2}$ since $0<\epsilon<\tfrac{1}{2}$):
\[ f \colon (2\epsilon^2,\tfrac{1}{2})\rightarrow (-\tfrac{1}{2},-2\epsilon^2), \qquad f(x)=-\epsilon^2 x^{-1}.\]
Then $f(\epsilon)=-\epsilon$ so $f$ may be used to glue together $M_{\ge 2\epsilon^2}$, $M_{\le -2\epsilon^2}$ along the collars
\[ C_1=Z \times (2\epsilon^2,\tfrac{1}{2}) \subset M_{\ge 2\epsilon^2}, \qquad C_2=Z \times (-\tfrac{1}{2},-2\epsilon^2) \subset M_{\le -2\epsilon^2},\]
giving a smooth manifold $\ti{M}$. Let
\[ q \colon M_{2\epsilon^2} \rightarrow \ti{M} \]
be the smooth gluing map, which identifies pairs of points in a collar neighborhood of the boundary of $M_{2\epsilon^2}$ according to the map $f\times \id_Z$ defined above.

Combining $f$ with the gluing map $\tau$ on the spinor bundle, we obtain a smooth version $\tau_{\tn{sm}}$ of $\tau$ that may be used in a smooth construction of $\ti{S}$.  Under $f$ the operator $x\partial_x$ goes to $-x\partial_x$.  Together with \eqref{eqn:tilDonC} and the fact that $\tau$ anti-commutes with $c(e)$ and commutes with $D_Z$ (see \eqref{eqn:anticommute}), we deduce
\[ \tau_{\tn{sm}}\circ \ti{D}=\ti{D} \circ \tau_{\tn{sm}} \]
which verifies that the operator $\ti{D}$ descends to a Dirac operator on $\ti{M}$ for $\ti{S}$.  

We now invoke a special case of the Splitting Theorem \cite[Theorem 8.17]{BarBallmann} (or see \cite[Theorem 4.5]{BarBallmannGuide}) for elliptic boundary value problems, which we briefly recall now.  Given a first-order essentially self-adjoint elliptic operator $\D$ acting on sections of $E$ on a closed manifold $\ti{M}$, we may cut $\ti{M}$ along a hypersurface $N$ to obtain a manifold $\ti{M}'$ with boundary $N_1\sqcup N_2$, and operator $\D'$.  If $B_1 \subset H^{1/2}(N,E)$ is an elliptic boundary condition with $L^2$-orthogonal complement $B_2 \subset H^{1/2}(N,E)$ then 
\[ \ind(\D)=\ind(\D'_{B_1\sqcup B_2}) \] 
where $\D'_{B_1\sqcup B_2}$ denotes the operator with domain specified by the boundary condition $B_1\sqcup B_2$ along $N_1\sqcup N_2=\partial \ti{M}'$.  For the case at hand we set $E=\ti{S}^+$, $N=q(Z_\epsilon)=q(Z_{-\epsilon})$, $\D=\ti{D}^+$. We may identify $\ti{M}'$ with $M_\epsilon$, and $N_1=Z_{\epsilon}$, $N_2=Z_{-\epsilon}$.  We take the boundary condition $B_1$ to be
\begin{equation} 
\label{eqn:defB1}
B_1=H^{1/2}(N,\ti{S}^+)_{A_\epsilon \le 0}=H^{1/2}(Z_\epsilon,S^+)_{A_\epsilon \le 0} 
\end{equation}
This corresponds to the `dual' APS boundary condition along $N_1=Z_\epsilon$, used in the Definition \ref{def:DAPS} of $D^+_{APS,\epsilon}$. Its $L^2$-orthogonal complement is
\[ B_2=H^{1/2}(N,\ti{S}^+)_{A_\epsilon>0}.\]
To compare with the boundary conditions used in Definition \ref{def:DAPS}, $B_2$ must be identified with a subspace of $H^{1/2}(Z_{-\epsilon},S^+)$.  By construction of $\ti{S}$, the identification $H^{1/2}(N,\ti{S}^+)\simeq H^{1/2}(Z_{-\epsilon},S^+)$ is given by the gluing map $\tau$.  By equation \eqref{eqn:anticommute}, $A_\epsilon\circ \tau=-\tau\circ A_{-\epsilon}$, hence $B_2$ identifies with
\[ B_2=H^{1/2}(Z_{-\epsilon},S^+)_{A_{-\epsilon}<0},\]
and this corresponds to the APS boundary condition along $Z_{-\epsilon}$ used in the Definition \ref{def:DAPS} of $D^+_{APS,\epsilon}$.  This verifies that $\ti{D}^+$, $D^+_{APS,\epsilon}$ satisfy the hypotheses of the Splitting Theorem, and thus
\[ \ind(\ti{D}^+)=\ind(D^+_{APS,\epsilon}).\]
Finally, since $\ti{M}$ is closed, the index of $\ti{D}$ is invariant under homotopy, and does not depend on $\epsilon$. 
\end{proof}

\begin{theorem}
For any  $b$-symplectic manifold $(M, \omega, Z)$, we have that 
\begin{equation}
\label{eqn:AtiySing}
Q(M, Z, \omega) = \int_M \widehat{A}(TM) \mathrm{Ch}(\ti{\L})^{1/2}=\int_M \tn{Td}(^bTM)\Ch(L)=\lim_{\epsilon \rightarrow 0^+} \int_{M_\epsilon} \tn{Td}(^bTM)e^\omega. 
\end{equation}
\end{theorem}
\begin{proof}
The first expression is the Atiyah-Singer formula for $\ind(\ti{D}^+)$.  One has $\widehat{A}(TM)=\widehat{A}(^bTM)$ (see Remark \ref{rem:gluingotherstructures}), and $\wedge{}^bT^{\ast,(0,1)}M$, $\ti{S}$ have isomorphic determinant line bundles (see Theorem \ref{thm:spincstr}), so the second expression follows from the first via the usual calculation relating the Hirzebruch-Riemann-Roch and spin-c versions of the Atiyah-Singer integrand.  The third expression follows from the second, together with the following general observation: if $\Omega$ is a top degree $b$-differential form and
\[ \Omega=\frac{dx}{x}\wedge \pi^\ast \Omega_Z'+\Omega'' \]
is a decomposition as in Lemma \ref{lem:decomp}, then
\[ \lim_{\epsilon\rightarrow 0^+}\int_{M_\epsilon} \Omega=\int_M \Omega''+\lim_{\epsilon \rightarrow 0^+} \int_{|x|\ge \epsilon} \frac{dx}{x} \int_Z \Omega_Z'=\int_M \Omega''+0\cdot \int_Z \Omega_Z'=\int_M \Omega''.\]
\end{proof}

\section{The equivariant case: quantization commutes with reduction}\label{S:qr}

In this section we consider a Hamiltonian action of a compact connected Lie group $G$ on $(M,Z,\omega)$. We define the geometric quantization $Q_G(M,Z,\omega)$ to be the equivariant index of the APS-like boundary value problem for $D^+_{\APS}$.  Under the assumption that the action of $G$ has non-zero modular weights, we show that $Q_G(M,Z,\omega)$ satisfies the Guillemin-Sternberg ``quantization commutes with reduction'' property.  In particular our quantization coincides with the formal geometric quantization defined in \cite{guillemin2018geometric}, and thus provides a positive answer to a question posed in \emph{loc. cit.}

\def\g{\ensuremath{\mathfrak{g}}}

\subsection{Hamiltonian group actions on a $b$-symplectic manifolds}\label{SS:hamiltonian}
Suppose that a compact connected Lie group $G$ acts on $M$ preserving $(Z,\omega)$.  Let $\kg$ denote the Lie algebra of $G$ and let $\kg^*$ be its dual. For $X\in \kg$ we denote by 
\begin{equation} 
\label{eqn:infaction}
X_M(m)=\frac{d}{dt}\bigg|_{t=0}\exp(-tX)\cdot m
\end{equation}
the vector field on $M$ generated by the infinitesimal action of $X$ on $M$, and likewise $X_Z$ for the restriction of $X_M$ to $Z$. More generally if $X$ is a smooth map to $\g$, then the same expression \eqref{eqn:infaction} but with $X$ evaluated at $m$ defines a vector field on $M$ still denoted by $X_M$.  Let $Z= \sqcup_{j=1}^k Z_j$ where the $Z_j$'s are connected components of $Z$. Since $G$ is assumed connected, each $Z_j$ is preserved by $G$.  

Let $x$ be a $G$-invariant global defining function for $Z$ as in Lemma \ref{lem:prodformomega}, with $|x|=1$ outside a $G$-invariant collar neighborhood $C\simeq Z \times (-1,1)_x$, and let $\pi \colon C \rightarrow Z$ be the $G$-equivariant projection map to $Z$. Recall that the $b$ symplectic form admits a decomposition
\begin{equation} 
\label{eqn:normalform1}
\omega=\frac{dx}{x}\wedge \pi^\ast \alpha_Z+\beta,
\end{equation}
where $\alpha_Z$ is a closed 1-form on $Z$ that is independent of any choices, hence is automatically $G$-invariant (hence so is $\beta$).
\begin{definition}
\label{def:modularweight}
Let $c \colon Z \rightarrow \g^\ast$ be the smooth function defined by
\[ \pair{c}{X}=\iota(X_Z)\alpha_Z, \qquad X \in \g.\]
Since $\alpha_Z$ is $G$-invariant and closed, the Cartan formula implies that $c$ is locally constant.  Then, $G$-invariance of $\alpha_Z$ implies that $c$ takes values in $\mf{z}^\ast \subset \g^\ast$ where $\mf{z}\subset \g$ is the center.  We call $c$ the \emph{modular weight function}, or just the \emph{modular weight} for short.
\end{definition}

\begin{definition}\label{D:hamiltonian}
We say that the action of $G$ on a $b$-symplectic manifold $(M,Z,\omega)$ is {\em Hamiltonian} if there exists a $G$-equivariant {\em moment map} $\mu\colon M\backslash Z\to \kg^*$ such that
\begin{enumerate}
\item $d\pair{\mu}{X}=-\iota(X_M)\omega$,
\item the function $\ol{\mu}:=\mu-\log(|x|)\pi^\ast c$ extends smoothly to $M$.
\end{enumerate}
Note that by assumption, outside the collar $C$, $|x|=1 \Rightarrow \log(|x|)=0$, hence $\log(|x|)\pi^\ast c$ can be viewed as a $\mf{z}^\ast$-valued function on $M$. Comparing with \eqref{eqn:normalform1}, we see $d\pair{\ol{\mu}}{X}=-\iota(X_M)\beta$, i.e. $\ol{\mu}$ is an ordinary moment map for the presymplectic form $\beta$.
\end{definition}

\subsection{$[Q,R]=0$ conjecture for $b$-symplectic manifolds}
Let $(M,Z,\omega,\mu)$ be a $b$-symplectic manifold with a Hamiltonian $G$-action, where $G$ is assumed to be compact and connected.  Suppose $M$ is prequantizable, and let $L$ be a prequantum line bundle.  Lift the infinitesimal $\g$ action on $M$ to $L$ via the Kostant formula: 
\[ \L_X=\nabla^{L,0}_{X_M}+2\pi \i \pair{\ol{\mu}}{X}, \qquad X \in \g.\]
If we use the definitions $\nabla^L_{X_M}=\nabla^{L,0}_{X_M}-2\pi \i \log(|x|)\pi^\ast\alpha_Z(X_Z)$, $\alpha_Z(X_Z)=\pair{c}{X}$ and $\ol{\mu}=\mu-\log(|x|)\pi^\ast c$, we see that over $M \backslash Z$ this is equivalent to
\[ \L_X=\nabla^L_{X_M}+2\pi \i \pair{\mu}{X}, \qquad X \in \g.\]
We furthermore assume that the $\g$-action integrates to a $G$-action. Proceeding as in Section \ref{SS:bc}, the operator $D^+_{APS,\epsilon}$ is $G$-equivariant and Fredholm, hence has a $G$-equivariant index lying in the character ring $R(G)$, and as in Definition \ref{D:quantizationregular}, we define
\begin{equation} 
\label{eqn:QG}
Q_G(M,Z,\omega)=\ind_G(D^+_{APS,\epsilon}) \in R(G).
\end{equation}

The constructions in Section \ref{SS:spinor} now lead to a $G$-equivariant spinor bundle $\ti{S}\rightarrow \ti{M}$, and Theorem \ref{thm:indexequal} goes through:
\[ Q_G(M,Z,\omega)=\ind_G(\ti{D}) \]
where $\ti{D}$ is any $G$-equivariant spin-c Dirac operator for $\ti{S}$.  As before it follows in particular that \eqref{eqn:QG} is independent of $\epsilon$, and one could equivalently define $Q_G(M,Z,\omega)$ in terms of a limit as $\epsilon \rightarrow 0^+$.

\begin{conjecture}\label{C:qr}
If $0$ is a regular value of the moment map and if the reduced space 
\[
M_\text{red} = \mu^{-1}(0)/G, \quad Z_{\text{red}} =( \mu^{-1}(0)\cap Z) /G
\]
is a $b$-symplectic manifold (orbifold), then
\[
[Q_G(M, Z, \omega)]^G = Q(M_{\text{red}}, Z_{\text{red}},  \omega_{\text{red}}) \in \bZ. 
\]
\end{conjecture}
\begin{remark}
For the case $Z = \emptyset$, this is the ``quantization commutes with reduction'' conjecture of Guillemin-Sternberg \cite{GuilleminSternbergConjecture}, first proved by Meinrenken \cite{MeinrenkenSymplecticSurgery}.	In case $0$ is not a regular value, one should modify the definition of the quantization of the reduced space using a shift desingularization for example, as in \cite{MeinrenkenSjamaar}.
\end{remark}

\subsection{The Kirwan vector field}
Fix an $\Ad_G$-invariant inner product $\pair{\cdot}{\cdot}_\g$ on $\g$.  Composing $c \colon Z \rightarrow \mf{z}^\ast$ with the isomorphism $\mf{z}^\ast \rightarrow \mf{z}$ induced by the inner product, we obtain a locally constant map
\[ \zeta \colon Z \rightarrow \mf{z}.\]
Since $\mf{z}$ is the center of $\g$, this map is trivially $G$-equivariant.  The induced vector field $\zeta_Z$ on $Z$ satisfies
\begin{equation} 
\label{eqn:normalizezeta}
\alpha_Z(\zeta_Z)=\pair{c}{\zeta}=|\zeta|_\g^2.
\end{equation}
Similarly we define
\[ \nu \colon M\backslash Z \rightarrow \g, \qquad \ol{\nu}\colon M \rightarrow \g \]
to be the compositions of $\mu$, $\ol{\mu}$ with the isomorphism $\g^\ast \rightarrow \g$.  The induced vector fields are all $G$-invariant and satisfy
\begin{equation} 
\label{eqn:nuzeta}
\nu_{M\backslash Z}=\ol{\nu}_{M\backslash Z}+\log(|x|)(\pi^\ast \zeta)_{C\backslash Z}.
\end{equation}
The vector field $\nu_{M \backslash Z}$ is the \emph{Kirwan vector field} (cf. \cite{Kirwan}) of the Hamiltonian $G$-space $(M\backslash Z,\omega,\mu)$.

\subsection{Proof of $[Q,R]=0$ for case of non-zero modular weights}\label{SS:qr}
In this section we prove quantization commutes with reduction in the following special case.
\vspace{0.3cm}

\noindent \textbf{Assumption}: The modular weight function $c$ is nowhere zero.  (In fact by \cite[Theorem 16]{guillemin2014convexity}, it already suffices to assume $c$ does not vanish on a single component of $Z$.)
\vspace{0.3cm}

\noindent Three immediate consequences of the assumption are:
\begin{enumerate}
\item All reduced spaces are symplectic (if non-singular).
\item The moment map $\mu \colon M \backslash Z \rightarrow \g^\ast$ is proper.
\item The vanishing set of the Kirwan vector field $\nu_{M\backslash Z}$ is a compact subset of $M \backslash Z$.
\end{enumerate}
The third statement follows from \eqref{eqn:nuzeta}, because under our assumption, $\zeta_Z$ is a non-vanishing vector field on $Z$, and the $\log(|x|)(\pi^\ast\zeta)_{C\backslash Z}$ term dominates as $x \rightarrow 0$ (the vector field $\ol{\nu}_{M\backslash Z}$ is bounded, since it extends smoothly to the vector field $\ol{\nu}_M$ on the compact manifold $M$).  These three consequences allow us to take advantage of results of Ma-Zhang \cite{MaZhangTransEll} on geometric quantization for proper moment maps on non-compact symplectic manifolds.  We will state the special case of their result that suffices for our purposes shortly.

\begin{lemma}
\label{lem:eprime}
We may choose $x$, $\pi$, $C$ to be $G$-equivariant and satisfy the conditions in Lemma \ref{lem:prodformomega}, and such that the resulting form $\beta=\omega-(dx/x)\wedge \pi^\ast \alpha_Z$ in \eqref{eqn:normalform1} satisfies the additional condition $\iota(\zeta_Z)(\iota^\ast \beta)=0$, where recall $\iota \colon Z \hookrightarrow M$.  Hence $\zeta_Z=|\zeta|_\g^2 e'$ in the notation of Lemma \ref{lem:prodformomega}.
\end{lemma}
\begin{proof}
Since $\zeta$ is central, $x\exp(-\pair{\ol{\mu}}{\zeta})$ is a $G$-invariant global defining function for $Z$ in the same equivalence class of global defining functions as $x$.  By re-scaling or using a bump function, we may find a $G$-invariant global defining function $x'$ that equals $x\exp(-\pair{\ol{\mu}}{\zeta})$ on a small neighborhood, say $C''$, of $Z$, and equals $x$ outside of $C$. Then on $C''$ we have:
\[ \omega=\frac{dx}{x}\wedge \pi^\ast \alpha_Z+\beta=\frac{dx'}{x'}\wedge \pi^\ast \alpha_Z+\beta', \qquad \beta'=
-d\pair{\ol{\mu}}{\zeta}\wedge \pi^\ast \alpha_Z+\beta.\]
Notice that
\[ \iota(\zeta_Z)(\iota^\ast \beta')=\iota^\ast d\pair{\ol{\mu}}{\zeta}+\iota(\zeta_Z)(\iota^\ast \beta)=0,\]
by the moment map property for $\ol{\mu}$.  This shows it is possible to modify the global defining function on a small neighborhood of $Z$ in order to obtain a $\beta'$ satisfying the additional property in the statement of the lemma.  So we may assume from the beginning that $x$ is chosen such that $\beta=\omega-(dx/x)\wedge \pi^\ast \alpha_Z$ has this property.  

The construction in the proof of Lemma \ref{lem:prodformomega} does not modify the pullback of $\beta$ to $Z$, since the diffeomorphism $\varphi$ fixes $Z$.  Hence after carrying out the modification of $x$, $C$, $\pi$ as in the proof, the new form $\ti{\beta}=\omega-(d\ti{x}/\ti{x})\wedge \ti{\pi}^\ast \alpha_Z$ will still have the desired property.  

The vector field $e'$ on $Z$ was defined as the unique vector field satisfying $\alpha_Z(e')=1$, $\iota^\ast\beta(e')=0$, so $\zeta_Z=|\zeta|_\g^2 e'$ follows from \eqref{eqn:normalizezeta}.
\end{proof}

Using the lemma, the operators $A_y$, $y=\pm \epsilon$ defining the APS-like boundary condition (see \eqref{eqn:Ax}) can be written as
\[ A_y=c(e)^{-1}B_y,\]
where $B_y$ is the operator
\begin{equation}
\label{eqn:defBeps} 
B_y=D_Z-\i k\log(|y|)c(\zeta_Z),
\end{equation}
and $k=2\pi|\zeta|^{-2}_\g$ is an (unimportant) locally constant positive normalization factor. 

\begin{definition}
Define a 2-parameter family $0\ne y \in (-\tfrac{1}{2},\tfrac{1}{2})$, $T\in \bR$ of operators on $Z_y$ by
\[ A_{y,T}=c(e)^{-1}B_{y,T}, \quad B_{y,T}=B_y-\i T c(\nu_{Z_y})=D_Z-\i(k+T)\log(|y|)c(\zeta_Z)-\i T c(\ol{\nu}_{Z_y}).\]
Let $A_{y,T}^G$, $B_{y,T}^G$ denote the restrictions to the subspace of $G$-invariant spinors.
\end{definition}

Let $D^+_{APS,\epsilon,T}$ be the Hilbert space operator defined as in Definition \ref{def:DAPS}, but using APS boundary conditions with the operators $A_{\pm \epsilon,T}$ in place of $A_{\pm \epsilon}$ (this is legitimate, because $A_{y,T}$ is essentially self-adjoint and has the same symbol as $A_{y}$, hence can also serve as an adapted boundary operator, cf. \cite[Section 2.2]{BarBallmannGuide}).  

The following is a special case of results of Ma-Zhang on geometric quantization for proper moment maps on non-compact symplectic manifolds.
\begin{theorem}[\cite{MaZhangTransEll}]
\label{thm:MaZhang}
Fix $\epsilon>0$ such that $M_\epsilon$ contains the vanishing locus of the Kirwan vector field $\nu_{M \backslash Z}$ in its interior.  There is a constant $T_0$ such that for $T>T_0$, 
\[ \ind_G(D^+_{APS,\epsilon,T})^G=Q(M_{\tn{red}},\omega_{\tn{red}}).\]
(If $0$ is a regular value, $(M_{\tn{red}},\omega_{\tn{red}})$ is a symplectic orbifold.  Otherwise the right-hand-side should be defined by a shift desingularization, as in \cite{MeinrenkenSjamaar}.)
\end{theorem}
\begin{remark}
It is in fact only necessary that $M_\epsilon$ contain the component $\mu^{-1}(0)$ of the vanishing locus of $\nu_{M \backslash Z}$, but we have chosen to state the result in this way because a slight variant holds for the multiplicity of any irreducible representation $\lambda \in \Lambda_+$ of $G$ in $\ind_G(D^+_{APS,\epsilon,T})$, with $M_{\tn{red}}$ replaced with the reduced space at $\lambda$, and the constant $T_0=T_0(\lambda)$ depending on $\lambda$.
\end{remark}
\begin{remark}
In their paper \cite{MaZhangTransEll}, Ma and Zhang in addition deform the differential operator on $M_\epsilon$ by a $0^{th}$-order term $\i T c(\nu_{M_\epsilon})$.  Because $M_\epsilon$ is a compact subset of $M \backslash Z$, this is a continuous deformation by a family of bounded operators and does not affect the index.  (On the other hand deforming the boundary conditions usually does change the index.)
\end{remark}
\begin{remark}
Theorem \ref{thm:MaZhang} is rather more sophisticated than necessary, because our manifold $M \backslash Z$ has relatively simple geometry at infinity.  But it allows us to give a rather quick proof, and sets our geometric quantization for $b$-symplectic manifolds into a broader context.  Another treatment of a version of Ma-Zhang's argument can be found in \cite{LSWittenDef}.
\end{remark}

\begin{lemma}
\label{lem:invertibility}
There is a constant $\epsilon_0>0$ such that for $0<|y|<\epsilon_0$, the operator $B_{y,T}^G$ is invertible for all $T\ge 0$.
\end{lemma}
\begin{proof}
The square $B_{y,T}^2$ is given by a formula of the form
\[ D_Z^2+T^2|\ol{\nu}_{Z_y}|^2+\log(|y|)\Big(T\ell_y+(k+T)^2\log(|y|)|\zeta_Z|^2\Big) \]
where $\ell_y$ is a first-order differential operator containing the cross-terms.  An observation of Tian and Zhang \cite[Theorem 1.6, Remark 1.9]{TianZhang}, is that the restriction $\ell_y^G$ to the $G$-invariant spinors is a $0^{th}$ order operator (a bundle endomorphism), and in our case it is also bounded uniformly in $y$ (it extends smoothly to $y=0$ for example, the explicit $\log(|y|)$'s in the expression above being the only singularities).  Note also that $\zeta_Z$ is non-vanishing, by our assumption that the modular weights are non-zero.  Moreover $k+T>0$ as $k>0$ and $T \ge 0$. It follows that when $\log(|y|)$ is sufficiently negative (i.e. for $|y|$ sufficiently small), the second term $(k+T)^2\log(|y|)|\zeta_Z|^2$ in the brackets dominates, and the operator $(B_{y,T}^2)^G$ is strictly positive, hence $B_{y,T}^G$ is invertible.
\end{proof}

\begin{corollary}
\label{cor:independenceofT}
For $0<\epsilon<\epsilon_0$, $\ind(D^+_{APS,\epsilon,T})^G=\ind(D^+_{APS,\epsilon})^G$ for all $T\ge 0$.
\end{corollary}
\begin{proof}
It is known that for a suitable smooth 1-parameter family $\A_T:=A_{\epsilon,T}$ specifying APS-like boundary conditions, the index only varies at values of $T$ such that $0 \in \tn{Spec}(\A_T)$, cf. \cite[Section 8.2]{BarBallmann}, \cite[Proposition 1.1]{MaZhangTransEll}, \cite[Corollary 5.3]{LSWittenDef} for related arguments.  The result is thus a consequence of Lemma \ref{lem:invertibility}.
\end{proof}

\noindent Corollary \ref{cor:independenceofT} and Theorem \ref{thm:MaZhang} now imply Conjecture \ref{C:qr}, in the special case of non-zero modular weights:
\begin{corollary}
$Q_G(M,Z,\omega)^G=Q(M_{\tn{red}},\omega_{\tn{red}})$.
\end{corollary}

In particular this proves that $Q_G(M,Z,\omega)$ coincides with the formal geometric quantization defined in \cite{guillemin2018geometric}.

\appendix

\section{`Coordinate-free' description of the spin-c structure}
In this appendix we give a conceptual construction of the spin-c structure on an (oriented) b-symplectic manifold which illuminates some of its properties.  We begin with some brief informal motivation. 

Our discussion is based on the following simple observation.  If $M$ is a manifold with boundary $\partial M=Z$, then it is well-known that ${}^bTM$ and $TM$ are non-canonically isomorphic.  To obtain an isomorphism, let $C\simeq Z \times [0,1]$ be a collar neighborhood of the boundary $\partial M=Z\times \{0\}$, and let $\pi$ (resp. $x$) denote the projection to $Z$ (resp. $[0,1]$).  Let $e=x\partial_x$ viewed as a section of ${}^bTM|_C$ (it restricts to the canonical section \eqref{eqn:canontriv} along $Z$).  Then ${}^bTM|_C\simeq \pi^\ast TZ \oplus \bR \cdot e$ and $TM|_C\simeq \pi^\ast TZ \oplus \bR \cdot \partial_x$, so one obtains an isomorphism ${}^bTM|_C\simeq TM|_C$ sending $e$ to $\partial_x$.  This isomorphism can be extended to $M$ using the isomorphism provided by the anchor map $a \colon {}^bTM|_{M\backslash C}\rightarrow TM|_{M\backslash C}$, because the latter also maps $e|_{Z\times \{1\}}=1\cdot \partial_x$ to $\partial_x$. (The resulting map is only continuous, but one easily obtains a smooth version using a partition of unity.)

If $Z$ lies in the interior of $M$ this no longer works: in this case the collar $C\simeq Z \times [-1,1]_x$, and to be able to extend continuously using the anchor map at $x=-1$, $e|_{Z\times \{-1 \}}$ would have to go to $-\partial_x$, and the differing signs causes a problem.  Instead of an isomorphism, $TM$, ${}^bTM$ are related (topologically) by a clutching construction involving a reflection.  Let $x$ be a global defining function for $Z$ and let $M_{\ge 0}=x^{-1}[0,\infty)$ (resp. $M_{\le 0}=x^{-1}(-\infty,0]$). Then $M$ is obtained by gluing $M_{\ge 0}$ to $M_{\le 0}$ along $Z=x^{-1}(0)$.  And, up to homotopy, ${}^bTM$ is obtained by gluing $TM_{\ge 0}$ to $TM_{\le 0}$ using the map which sends $(m,v)\in TM_{\ge 0}|_Z$ to $(m,rv)\in TM_{\le 0}|_Z$, where $r$ is fibrewise orthogonal reflection in the subbundle $TZ \subset TM|_Z$ (with respect to any metric on $TM|_Z$).  A simple example is $M=S^1$, $Z=\{\pt\}$, when ${}^bTM$ is the non-orientable $\bR$ line bundle over $S^1$.

Below we will give a more systematic discussion using principal bundles; the discussion above is recovered by passing to associated bundles.  Throughout $Z \subset M$ is a hypersurface, $M$ is connected, and for convenience we equip $TM$, $^bTM$ with positive definite metrics.

\subsection{Surgery of principal bundles.}
Let $G$ be a Lie group, and $P \rightarrow M$ a principal $G$-bundle.  Let $\Ad(P)=P\times_{\Ad} G$ be the adjoint bundle.  Gauge transformations of $P$ are automorphisms of $P$ covering the identity on $M$, and are in one-one correspondence with global sections of the Lie group bundle $\Ad(P)$; we write $\scr{G}(P)$ for the space of global sections of $\Ad(P)$.

Let $q\colon M' \rightarrow M$ be the manifold with boundary obtained by cutting $M$ along $Z$.  The boundary of $M'$ comes equipped with a smooth involution $j \colon \partial M' \rightarrow \partial M'$.  If the normal bundle to $Z$ is orientable, then $\partial M'=Z_1\sqcup Z_2$ consists of two disjoint copies of $Z$, and $j$ swaps the two copies. Note that points in the pullback bundle $q^\ast P$ are labelled by pairs $(m',p)$ where $m' \in M'$ and $p \in P_{q(m')}$. 
\begin{definition}
\label{def:Pg}
Let $g \in \scr{G}(P|_{Z})$ be an involution.  We define a new principal $G$-bundle $P^g$ over $M$ by
\[ P^g=q^\ast P/\sim \]
where the equivalence relation identifies points $(m',p)\sim (j(m'),g(p))$ for $m' \in \partial M'$ (this is an equivalence relation because $j$, $g$ are involutions).  We will refer to this clutching construction as \emph{surgery of $P$ along $Z$ using $g \in \scr{G}(P|_Z)$}.  One easily obtains a smooth version of this construction by choosing a tubular neighborhood $\pi \colon U \rightarrow Z$ and an isomorphism $P|_U\simeq \pi^\ast (P|_Z)$, and then one uses the pullback gauge transformation $g_U=\pi^\ast g$ to define the equivalence relation used for clutching along a collar neighborhood.  
\end{definition}

The construction is functorial: if $P_i \circlearrowleft G_i$ are principal bundles over $M$, $\rho \colon G_1 \rightarrow G_2$ is a group homomorphism intertwining a map $f \colon P_1 \rightarrow P_2$, and $g \in \scr{G}(P_1|_Z)$ then there is an induced $\rho(g) \in \scr{G}(P_2|_Z)$ and a map $f^g\colon P_1^g \rightarrow P_2^{\rho(g)}$. 

\begin{example}
\label{ex:PO}
Let $P_O(TM)$, $P_O(^bTM)$ denote the orthonormal frame bundles of $TM$, ${}^bTM$ respectively.  Note that $\Ad(P_O(TM))=O(TM)$ is the bundle of orthogonal groups of the tangent bundle, and similarly for $\Ad(P_O(^bTM))$. Then $P_O(^bTM)\simeq P_O(TM)^r$ where $r=g\in \scr{G}(P_O(TM|_Z))$ is fibrewise orthogonal reflection in the subbundle $TZ\subset TM|_Z$.
\end{example}

\begin{remark}
\label{rem:gluingotherstructures}
Structures on $P$, if invariant under $g \in \scr{G}(P|_Z)$, can be glued together to give structures on $P^g$.  For smoothness one should thicken $Z$ to a tubular neighborhood $\pi \colon U \rightarrow Z$ as in Definition \ref{def:Pg}, with an isomorphism $P|_U\simeq \pi^\ast (P|_Z)$, and require invariance under the pullback gauge transformation $g_U=\pi^\ast g$. For example, a connection on $P$, if $g_U$-invariant, induces a connection on $P^g$. One can always find $g_U$-invariant connections on $U$ by averaging over the two element subgroup $\{1,g_U\}$ of $\scr{G}(P|_U)$, and using bump functions to extend to $M$.  A corollary is that one can arrange that characteristic forms obtained from the Chern-Weil construction for $P$, $P^g$ are the same.
\end{remark}

\subsection{Graded principal bundles.}
Our use of the expression `graded group' does not seem to be very common, but appears for example in \cite{FreedMoore}.
\begin{definition}
A \emph{graded Lie group} $(G,\partial)$ is a Lie group $G$ together with a smooth homomorphism $\partial \colon G \rightarrow \bZ_2=\{1,-1\}$.  A \emph{graded principal $G$ bundle} $(P,\partial)$ is a principal $G$-bundle together with a smooth map $\partial \colon P \rightarrow \bZ_2$ such that $\partial(pa)=\partial p\cdot \partial a$ for all $a \in G$, $p \in P$.  The \emph{opposite grading} on $P$ is the map $-\partial$.
\end{definition}

\begin{example}
\label{ex:orient}
Let $V$ be a Euclidean vector bundle.  Orientations of $V$ are in one-one correspondence with gradings of the frame bundle $P_O(V)$.  The opposite grading corresponds to the opposite orientation.  In particular $^bTM$ with its symplectic orientation determines a grading $^b\partial_O$ on $P_O(^bTM)$.
\end{example}

If $(P,\partial)$ is graded and $m \in M$ then a gauge transformation $g \in \scr{G}(P_m)$, viewed as an $\Ad_G$-equivariant map $P_m\rightarrow G$, necessarily has constant degree $\partial g \in \bZ_2$ because $\bZ_2$ is abelian and $\partial \colon G \rightarrow \bZ_2$ is a group homomorphism.  Thus if $g \in \scr{G}(P|_Z)$ then $\partial g$ must be constant over each connected component of $Z$.  If $\partial g=1$ then $P^g$ is naturally graded again, the grading being obtained by following the construction in Definition \ref{def:Pg} and using $g$ to patch together the pullback grading on $q^\ast P$. With a global defining function, one can also obtain a grading in the case $\partial g=-1$:
\begin{definition}
\label{def:gradedsurgery}
Let $(P,\partial)$ be a graded principal $G$-bundle over $M$.  Let $Z \subset M$ be a hypersurface admitting a global defining function $x$.  Recall from Definition \ref{def:Pg} that $q \colon M'\rightarrow M$ denotes the quotient map from the manifold with boundary $M'$.  Let $M'=M_1\sqcup M_2$ where $M_1$ (resp. $M_2$) is the union of the components of $M'$ where $x\circ q \ge 0$ (resp. $\le 0$).  Let $Z_i=\partial M_i$. Let $g \in \scr{G}(P|_Z)$ with $\partial g=-1$.  Then we define a grading $\partial^{g,[x]}$ on $P^g$ by following the construction of $P^g$ but using the opposite grading over $M_2$; after reversing the grading over $M_2$, $g \in \scr{G}(P|_Z)$ can be viewed as a grading-preserving morphism $q^\ast P|_{Z_1}\rightarrow q^\ast P|_{Z_2}$ of graded principal bundles, and so the gradings patch together to a grading of $P^g$. The construction depends only on the equivalence class $[x]$ of the global defining function $x$, modulo multiplication by globally defined smooth positive functions; since $M$ is connected there are exactly two equivalence classes $[x]$ and $[-x]$, and $\partial^{g,[-x]}=-\partial^{g,[x]}$.
\end{definition}

\begin{example}
\label{ex:orient2}
Continuing Examples \ref{ex:orient}, \ref{ex:PO}, if $Z$ admits a global defining function $x$ then applying Definition \ref{def:gradedsurgery} to the graded bundle $(P_O(^bTM),{}^b\partial_O)$ and gauge transformation $r$, one obtains a grading ${}^b\partial_O^{r,[x]}=:\partial_O$ on $P_O(^bTM)^r\simeq P_O(TM)$.  Of course we already saw that in the $b$-symplectic case, the choice of an orientation on $M$ determines an equivalence class of global defining functions and conversely; the construction of $\partial_O$ reaffirms this.
\end{example}

\subsection{Spin-c structures.}
A sample reference for this subsection is Lawson-Michelsohn \cite{LawsonMichelsohn}. Let $(V,\pair{-}{-})$ be an $n$-dimensional Euclidean vector space over $\bR$.  Our convention for the real Clifford algebra $Cl(V)\simeq Cl(\bR^n)$ is $v_1v_2+v_2v_1=-2\pair{v_1}{v_2}$.  For convenience throughout we take $n \ge 2$ to be even. 

Recall that $\Pin(V)\simeq \Pin(n)$ is the set of products $a=v_1\cdots v_k \in Cl(V)$ where each $v_i \in V$; such elements are units in $Cl(V)$ with $a^{-1}=(-1)^k v_k\cdots v_1$, hence $\Pin(V)$ is a group, and in fact a graded group if we define $\partial a=(-1)^k$. The subgroup of even elements $\Spin(V)=\partial^{-1}(1)\simeq \Spin(n)$.  The complexification $\bC l(V)=Cl(V)\otimes \bC$ is an algebra over $\bC$. Define $\Pin_c(V)\simeq \Pin_c(n)$ as the set of products $zv_1\cdots v_k \in \bC l(V)$ where $z \in U(1)\subset \bC$ and $v_i \in V$. The grading extends to $\Pin_c(V)$ in the obvious fashion, and $\Spin_c(V)=\partial^{-1}(1)\simeq \Spin_c(n)$ is the subgroup of even elements. There is a well-defined map 
\[ \det \colon \Pin_c(V)\rightarrow U(1), \qquad zv_1\cdots v_k\mapsto z^2.\]

If $w \in V$ has length $1$, then the map
\[ a \in Cl(V) \mapsto \Ad_w(a)=waw^{-1}=-waw \in Cl(V) \]
preserves $V$ and restricts to an orthogonal transformation, namely reflection in the hyperplane perpendicular to $w$ composed with inversion $v\mapsto -v$.  The map $w \mapsto \Ad_w$ extends to a group homomorphism
\[ \Ad \colon \Pin_c(V)\rightarrow O(V),\]
which restricts to $2:1$ coverings $\Pin(V)\rightarrow O(V)$ and $\Spin(V)\rightarrow SO(V)$.  Choosing an orientation on $V$, the \emph{chirality element} is the product
\[ \Gamma=\i^{n/2} e_1\cdots e_n \in \Spin_c(V) \]
where $e_1,...,e_n$ is any oriented orthonormal basis of $V$.  It has the property $\Gamma^2=1$, and $\Ad_\Gamma$ implements the grading $\partial$.  In particular if $v \in V$ then $\Ad_\Gamma(v)=-v$, hence if $0\ne w \in V$, then $\Ad_{\Gamma w}$ is reflection in the hyperplane perpendicular to $w$.

The above discussion extends to Euclidean vector bundles $V$.  For example, $\Pin_c(V)$ denotes the bundle of groups $\Pin_c(V_m)$, $m \in M$.  The following definition is slightly unconventional but equivalent to the more usual ones.
\begin{definition}
A \emph{spin-c structure} on a rank $n$ Euclidean vector bundle $V$ is a grading $\partial_O$ on the orthonormal frame bundle $P_O(V)$ together with a graded principal $\Pin_c(n)$-bundle $(P,\partial)$ and a map of graded principal bundles $f\colon (P,\partial) \rightarrow (P_O(V),\partial_O)$ such that $f(pg)=f(p)\cdot \Ad_g$ for all $g \in \Pin_c(n)$, $p \in P$.  The \emph{determinant line bundle} (or \emph{anti-canonical line bundle}) of the spin-c structure is the associated bundle $P\times_{\det} \bC$.
\end{definition}
Note also that $\Ad(P)=P\times_\Ad \Pin_c(n) \simeq \Pin_c(V)$ identifies with the bundle of $\Pin_c$ groups of the fibres, and the map $\Pin_c(V)\rightarrow O(V)$ induced by $f$ is simply the map $\Ad$ discussed earlier.

We can now give a concise construction of a canonical (up to isomorphism) spin-c structure on an (oriented) b-symplectic manifold.
\begin{theorem}
\label{thm:spincstr}
The tangent bundle $TM$ of an oriented b-symplectic manifold has a canonical spin-c structure.  The determinant line bundle is isomorphic to the determinant line bundle of the spin-c structure on $^bTM$ coming from a compatible almost complex structure.
\end{theorem}
\begin{proof}
The orientations on $TM$, $^bTM$ determine an equivalence class of global defining functions $[x]$.  Let $r \in \Gamma(O(^bTM|_Z))=\scr{G}(P_O(^bTM|_Z))$ be orthogonal reflection in the subbundle $TZ\subset {}^bTM|_Z$.  Recall from Example \ref{ex:orient2} that applying the construction of Definition \ref{def:gradedsurgery} to $(P_O(^bTM),{}^b\partial_O)$ yields a graded bundle $(P_O(^bTM)^r,{}^b\partial^{r,[x]}_O)\simeq (P_O(TM),\partial_O)$ where $\partial_O$ is the grading coming from the orientation on $TM$.

The symplectic vector bundle $^bTM$ has a canonical spin-c structure $(^bP,{}^b\partial)\rightarrow (P_O(^bTM),{}^b\partial_O)$ coming from a compatible almost complex structure.  Recall the canonical non-vanishing section $e \in C^\infty(Z,{}^bTM|_Z)$ (see equation \eqref{eqn:canontriv}). We may assume the metric on $^bTM$ is such that $|e|=1$.  The orthogonal reflection $r$ has a canonical lift $\Gamma e \in \Gamma(\Pin_c(^bTM|_Z))=\scr{G}(^bP|_Z)$ (i.e. $\Ad_{\Gamma e}=r$), where the symplectic (or complex) orientation on $^bTM$ is used to define the chirality element $\Gamma$ for $\bC l(^bTM)$.  Note also that $(\Gamma e)^2=-\Gamma^2 e^2=1$ is an involution.  Applying the construction of Definition \ref{def:gradedsurgery} to $g=\Gamma e$ and the class $[x]$, we obtain a principal $\Pin_c(n)$-bundle with grading $(P,\partial):=(^bP^g,{}^b\partial^{g,[x]})$.  By functoriality of surgery, there is a map of graded bundles 
\[ (P,\partial)\rightarrow (P_O(^bTM)^r,{}^b\partial^{r,[x]}_O)\simeq (P_O(TM),\partial_O).\] 
Hence we obtain a spin-c structure for $TM$.  

Since $\det(\Gamma e)=\i^n$ ($=\pm 1$), the induced surgery on the associated line bundle $^bP\times_{\det} \bC$ is the constant automorphism given by multiplication by $\i^n$ along $Z$, but this does not change the topological type of a complex line bundle (it is homotopic to the constant map $1$ in $U(1)$).
\end{proof}
\begin{remark}
The second Stieffel-Whitney classes $w_2(^bTM)=w_2(TM)$ in $H^2(M,\bZ_2)$ (and this holds more generally, with $M$ not required to be even-dimensional or oriented).  One can see this, for example, by constructing a C{\v e}ch cocycle, using the fact that $w_2(V)$ is the obstruction to lifting transition funtions $t_{ij} \colon U_{ij}\rightarrow O(n)$ for $V$ to $\Pin(n)$.  Consequently, if $TM$, $^bTM$ are both oriented, then $TM$ admits a Spin(-c) structure if and only if $^bTM$ admits a Spin(-c) structure.
\end{remark}
\begin{remark}
One feature of the above construction is that it can be applied recursively in the more general situation where $Z$ is allowed to have normal crossing singularities.  In this case $Z=\cup Z_j$ is a union of smooth closed connected hypersurfaces, and the $Z_i$ are allowed to intersect, but at any intersection point $m$, the tangent spaces $T_mZ_i$ look like a subset of the coordinate hyperplanes in $\bR^n$ (see \cite{gualtieri2017tropical} for details).  Assume $^bTM$ is equipped with a metric such that the canonical sections $e_j$ of $^bTM|_{Z_j}$ are orthonormal at the intersection points. Up to homotopy, the frame bundle $P_O$ can be obtained from $^bP_O$ by successive surgeries along $Z_1,Z_2,...$.  Assuming each $Z_j$ admits a global defining function, then one obtains a grading (or orientation, see Example \ref{ex:orient}) on each of the intermediate stages $^bP_O^{r_1}$, $(^bP_O^{r_1})^{r_2}$,...,$P_O$ using Definition \ref{def:gradedsurgery}.  If $^bP$ is the spin-c structure coming from the complex structure on $^bTM$, then $^bP^{g_1}$, $(^bP^{g_1})^{g_2}$,... are spin-c structures for $^bP_O^{r_1}$, $(^bP_O^{r_1})^{r_2}$,..., where $g_j=\Gamma_j e_j$ and $\Gamma_j$ is the chirality element defined using the orientation at the $(j-1)^{st}$ stage (note that, for example, $e_2$ canonically determines an element of $\scr{G}(^bP^{g_1}|_{Z_2})$ because $\Ad_{\Gamma_1 e_1}(e_2)=e_2$ is fixed by the gauge transformation along $Z_1 \cap Z_2$).  We will discuss this further elsewhere. 
\end{remark}

Up to equivalence, the algebra $\bC l(n)$ has a unique irreducible representation $\Delta \simeq \bC^{2^{n/2}}$,
\[ \delta \colon \bC l(n)\xrightarrow{\sim} \End(\Delta).\]
Moreover $\delta$ maps $\Pin_c(n)$ into a subgroup of the unitary group $U(\Delta)$.  The element $\delta(\Gamma)$ squares to $1$ so determines a grading on $\Delta$.  The spinor module associated to a spin-c structure $(P,\partial)$ for $V$ is the associated bundle
\[ S=P\times_{\delta} \Delta.\]
$S$ inherits a $\bZ_2$-grading from those on $P$ and $\Delta$.  $S$ also automatically inherits a Hermitian structure and $\bC l(V)=P\times_\Ad \bC l(n)$-module structure $c \colon \bC l(V)\xrightarrow{\sim} \End(S)$ from the inner product and $\bC l(n)$-module structure on $\Delta$.  Applying this construction to the spin-c structure in Theorem \ref{thm:spincstr} recovers the spinor module described in Section \ref{SS:spinor} (up to isomorphism).


\bibliographystyle{amsplain}

\begin{thebibliography}{10}

\bibitem{AtiyahPatodiSingerI}
M.F. Atiyah, V.K. Patodi, and I.M. Singer, \emph{Spectral asymmetry and
  {R}iemannian geometry. {I}}, Math. Proc. Cam. Phil. Soc., vol.~77, 1975,
  pp.~43--69.

\bibitem{BarBallmann}
W.~Ballmann and C.~B{\"a}r, \emph{Boundary value problems for elliptic
  differential operators of first order}, Surv. Diff. Geom. (H.-D. Cao and
  S.-T. Yau, eds.), vol.~17, Int. Press., 2012, pp.~1--78.

\bibitem{BarBallmannGuide}
\bysame, \emph{Guide to elliptic boundary value problems for {D}irac-type
  operators}, Arbeitstagung Bonn 2013, Springer, 2016, pp.~43--80.

\bibitem{FreedMoore}
D.~Freed and G.~Moore, \emph{Twisted equivariant matter}, Annales Henri
  Poincar{\'e} \textbf{14} (2013), no.~8, 1927--2023.

\bibitem{GromovLawson}
M.~{Gromov} and H.B. {Lawson}, \emph{Positive scalar curvature and the {D}irac
  operator on complete {R}iemannian manifolds}, Pub. Math. de l'IH{\'E}S
  \textbf{58} (1983), no.~1, 83--196.

\bibitem{gualtieri2014symplectic}
M.~{Gualtieri} and S.~{Li}, \emph{Symplectic groupoids of log symplectic
  manifolds}, Int. Math. Res. Not. \textbf{2014} (2014), no.~11, 3022--3074.

\bibitem{gualtieri2017tropical}
M.~Gualtieri, S.~Li, A.~Pelayo, and T.~Ratiu, \emph{The tropical momentum map:
  a classification of toric log symplectic manifolds}, Mathematische Annalen
  \textbf{367} (2017), no.~3-4, 1217--1258.

\bibitem{guillemin2014symplectic}
V.~{Guillemin}, E.~{Miranda}, and A.~R. {Pires}, \emph{Symplectic and {P}oisson
  geometry on b-manifolds}, Adv. Math. \textbf{264} (2014), 864--896.

\bibitem{guillemin2014convexity}
V.~{Guillemin}, E.~{Miranda}, A.~R. {Pires}, and G.~{Scott}, \emph{Convexity
  for {H}amiltonian torus actions on $ b $-symplectic manifolds}, Math. Res.
  Letters \textbf{24} (2017), 363--377.

\bibitem{GuilleminSternbergConjecture}
V.~{Guillemin} and S.~{Sternberg}, \emph{Geometric quantization and
  multiplicities of group representations}, Invent. Math. \textbf{67} (1982),
  no.~3, 515--538.

\bibitem{guillemin2018geometric}
V.~W {Guillemin}, E.~{Miranda}, and J.~{Weitsman}, \emph{On geometric
  quantization of b-symplectic manifolds}, Adv. Math. \textbf{331} (2018),
  941--951.

\bibitem{guillemin2014toric}
Victor Guillemin, Eva Miranda, Ana~Rita Pires, and Geoffrey Scott, \emph{Toric
  actions on b-symplectic manifolds}, Int. Math. Res. Not. \textbf{2015}
  (2014), no.~14, 5818--5848.

\bibitem{Kirwan}
F.~{Kirwan}, \emph{Cohomology of quotients in symplectic and algebraic
  geometry}, Princeton Univ. Press, 1984.

\bibitem{LawsonMichelsohn}
H.~B. {Lawson} and M-L. {Michelsohn}, \emph{Spin geometry}, Princeton Univ.
  Press, 1989.

\bibitem{LSWittenDef}
Y.~{Loizides} and Y.~{Song}, \emph{Witten deformation for {H}amiltonian loop
  group spaces}, ar{X}iv:1810.02347.

\bibitem{MaZhangTransEll}
X.~{Ma} and W.~{Zhang}, \emph{Geometric quantization for proper moment maps:
  the {V}ergne conjecture}, Acta Math. \textbf{212} (2014), no.~1, 11--57.

\bibitem{MeinrenkenSymplecticSurgery}
E.~{Meinrenken}, \emph{Symplectic surgery and the {S}pin-c {D}irac operator},
  Adv. Math. \textbf{134} (1998), 240--277.

\bibitem{MeinrenkenSjamaar}
E.~{Meinrenken} and R.~{Sjamaar}, \emph{Singular reduction and quantization},
  Topology \textbf{38} (1999), no.~4, 699--762.

\bibitem{melrose1993atiyah}
R.~Melrose, \emph{The {A}tiyah-{P}atodi-{S}inger index theorem}, AK Peters/CRC
  Press, 1993.

\bibitem{NestTsygan96}
R.~{Nest} and B.~{Tsygan}, \emph{Formal deformations of symplectic manifolds
  with boundary}, J. Reine Angew. Math. \textbf{481} (1996), 27--54.

\bibitem{TianZhang}
Y.~{Tian} and W.~{Zhang}, \emph{An analytic proof of the geometric quantization
  conjecture of {G}uillemin?-{S}ternberg}, Invent. Math. \textbf{132} (1998),
  229--259.

\bibitem{Weitsman01-QNC}
J.~{Weitsman}, \emph{Non-abelian symplectic cuts and the geometric quantization
  of noncompact manifolds}, Lett. Math. Phys. \textbf{56} (2001), 31--40.

\end{thebibliography}
\providecommand{\bysame}{\leavevmode\hbox to3em{\hrulefill}\thinspace}
\providecommand{\MR}{\relax\ifhmode\unskip\space\fi MR }
\providecommand{\MRhref}[2]{%
  \href{http://www.ams.org/mathscinet-getitem?mr=#1}{#2}
}
\providecommand{\href}[2]{#2}

\end{document}